\documentclass[12pt]{article}
\usepackage{amsmath}
\usepackage{bbm}
\usepackage{amsfonts}
\usepackage{amsmath, amssymb, amsthm}
\usepackage{mathdots}
\usepackage{titlesec, titletoc}
\usepackage{graphicx}
\usepackage{diagbox}
\usepackage{hyperref}

\newtheorem{lem}{Lemma}[section]%
\newtheorem{theorem}[lem]{Theorem}%
\newtheorem{prob}{Problem}%
\newtheorem{prop}[lem]{Proposition}%

\def\a{\alpha} \def\b{\beta}  \def\d{\delta} \def\S{\rm S}
    
 \def\ld{\lambda} 
 
 \def\O{\Omega}
\def\G{\Gamma}

\def\di{\bigm|} \def\lg{\langle} \def\rg{\rangle}
\def\nd{\mathrel{\bigm|\kern-.7em/}}

\def\f{\noindent}

\def\PSL{\hbox{\rm PSL}}
\def\SL{\hbox{\rm SL}}
\def\GL{\hbox{\rm GL}}
\def\PGL{\hbox{\rm PGL}}
\def\lcm{\hbox{\rm lcm}}
\def\PSp{\hbox{\rm PSp}}
\def\Aut{\hbox{\rm Aut\,}}

\def\Cay{\hbox{\rm Cay}}

\def\BiCay{\hbox{\rm BiCay}}

\def\mod{\hbox{\rm mod }}

\def\PSU{\hbox{\rm PSU}}
\def\GF{\hbox{\rm GF}}

\def\F{\hbox{\rm F}}
\def\A{\hbox{\rm A}}
\def\P{\hbox{\rm P}}

\def\R{{\mathcal R}}

\def\demo{{\bf Proof}\hskip10pt}

\def\mz{{\mathbb Z}}
\def\qed{\hskip10pt $\Box$\vspace{3mm}}

\def\mod{\hbox{\rm mod}}

\def\Mult{{\rm Mult}}

\begin{document}

\title{Cubic edge-transitive graphs of order $\ell p^2$ }
\renewcommand{\thefootnote}{\fnsymbol{footnote}}
\footnotetext[1]{Corresponding author. \\
E-mails: xwang@email.tjut.edu.cn(X. Wang), huangjj@cqupt.edu.cn(J.-J. Huang),  dwyang@bupt.edu.cn (D.-W. Yang)}

\author{ \\ Xue Wang$^{\rm a}$, Jun-Jie Huang$^{\rm b*}$, Da-Wei Yang$^{\rm c,d}$\\
{\small\em $^{\rm a}$School of Mathematical Science, Tianjin University of Technology, }\\
{\small\em Tianjin 300384, P.R. China}\\
{\small\em $^{\rm b}$School of Mathematics and Statistics, Chongqing University of Posts and Telecommunications,}\\
{\small\em Chongqing, 400065, P. R. China}\\
{\small\em $^{\rm c}$School of Mathematical Sciences, Beijing University of Posts and Telecommunications,}\\
{\small\em Beijing, 100876, P.R. China}\\
{\small\em $^{\rm d}$Key Laboratory of Mathematics and Information Networks (BUPT), Ministry of Education,}\\
{\small\em Beijing, 100876, P.R. China}
}

\date{}
\maketitle
\begin{abstract}
A graph is edge-transitive if its automorphism group acts transitively on its edge set.
This paper is devoted to a systematic study of cubic edge-transitive graphs of order $\ell p^2$, where $p$ is a prime and $\ell$ is an even integer.
We first show that either $\G$ is the unique cubic edge-transitive graph of order $18$, or its full automorphism group $\Aut(\G)$ contains a normal Sylow $p$-subgroup $P$ which is isomorphic to either $Z_{p^2}$ or $Z_{p}^2$.
For the case $P\cong Z_{p^2}$, we prove that $\G$ is either arc-regular or edge-regular,
and provide a complete classification of such graphs. For the case $P\cong Z_{p}^2$, we prove that $\G$ is a normal cover of certain known graphs. This work generalizes previous classifications by Zhou and Feng (2024), Lu et al.~(2004), and others.
As an application, we obtain a classification of cubic edge-transitive graphs of order $\ell p^2$
in the case where $\ell$ is square-free.

\end{abstract}
\section{Introduction}
All groups in this paper are finite, and all graphs are finite, simple and undirected.
Our focus is on graphs with specific symmetries, which are important subjects in algebraic graph theory and frequently serve as models for interconnection networks.

Let $\G$ be a graph. We use $V(\G)$, $E(\G)$, $A(\G)$ and $\Aut(\G)$ to denote the vertex set, edge set, arc set and the full automorphism group of $\Gamma$, respectively.
For a subgroup $G$ of $\Aut(\G)$, the graph $\G$ is said to be {\em $G$-vertex-transitive}, {\em $G$-edge-transitive} or {\em $G$-arc-transitive},
if $G$ is transitive on $V(\G)$, $E(\G)$ or $A(\G)$, respectively.
Subsequently, $\G$ is called {\em $G$-semisymmetric} if $\G$ is $G$-edge-transitive but not $G$-vertex-transitive.
Additionally, if $G$ acts regularly on $E(\G)$ or $A(\G)$, then $\G$ is said to be {\em $G$-edge-regular} or {\em $G$-arc-regular}, respectively.
In cases where $G=\Aut(\G)$, we shall dispense with the explicit mention of $G$ in the aforementioned definitions.
Moreover, an arc-transitive graph is also called a {\em symmetric graph}.

It is well-known that a connected graph of valency $2$ is just a cycle, which is trivially edge-transitive. In stark contrast, cubic edge-transitive graphs exhibit remarkable structural richness and have attracted considerable attention from researchers worldwide.

Notably, if a graph of odd valency is both vertex- and edge-transitive, it must be arc-transitive. Consequently, every cubic edge-transitive graph is either arc-transitive or semisymmetric. The study of these two classes was initiated by Tutte~\cite{Tutte1947,Tutte1959} for arc-transitive graphs and by Folkman~\cite{Folkman} for semisymmetric graphs, with subsequent contributions by Goldschmidt, Conder, Feng, and others. Classifying these graphs relies crucially on understanding their automorphism groups. A major achievement has been the characterization of the automorphism groups of cubic edge-transitive graphs. Specifically, the automorphism group of every cubic arc-transitive graph can be derived as a quotient of one of seven finitely presented groups (see~\cite{CL,DM}), while that of any finite cubic semisymmetric graph can be obtained as a homomorphic image of the universal completion of one of $15$ finite amalgams (see~\cite{Gold}). Based on these group-theoretic characterizations, all cubic arc-transitive and semisymmetric graphs on up to $10000$ vertices have been classified up to isomorphism, see~\cite{Conder-sym,Conder-semi-sym,CD,CMMP} for example.

As indicated earlier, one of the most active and appealing directions in the study of cubic edge-transitive graphs is their classification by order. A considerable body of work, for instance, has been devoted to classifying cubic edge-transitive graphs of order $\ell p$, where $\ell$ is a fixed integer and $p$ is a prime. In the arc-transitive case, by using voltage graphs and covering techniques, such graphs of order $\ell p$ with $2\leq \ell\leq 16$ have been classified for each prime $p$ (see~\cite{FK-2004,FK-AustMath,FK-SciChina,FK-JCTB,FKW,LCM,Oh14p,Oh16p}).
For the semisymmetric case,  results from~\cite{CO,DX} show that, apart from one graph of order $182$, all cubic edge-transitive graphs of order $\ell p$ with $\ell\in\{2, 4, 6, 10,14\}$ are arc-transitive. Furthermore, based on the findings in~\cite{AH,ALH,LCM,ZF2012}, except for a graph of order $112$, every cubic edge-transitive graph of order $\ell p$ with $\ell\in\{8, 12, 16,36\}$ is arc-transitive. More recently, Zhou and Feng~\cite{ZF2024} gave a characterization of connected cubic edge-transitive graphs of order $\ell p$ with $0<\ell<p$, and subsequently classified all connected cubic semisymmetric graphs of order $\ell$ with $0 < \ell \leq 54$.
In light of these developments, a natural and intriguing problem arises as follows:

\begin{prob}\label{question}
For a given integer $\ell$ and a prime $p$, characterize cubic edge-transitive graphs of order $\ell p^2$.
\end{prob}

Clearly, every cubic graph has even order, so in Problem~\ref{question}, $\ell$ is even and $\ell \neq p$ when $p\geq3$. In most cases, cubic edge-transitive graphs of order $\ell p^2$ with $\ell>p$ are already well-documented, especially when $\ell p^2\leq 10000$. Therefore, the most interesting direction for studying Problem~\ref{question} lies in the case where $\ell < p$.
Several studies have addressed this problem. For example, classifications of cubic arc-transitive graphs of order $\ell p^2$ for $\ell\in \{2,4,6,8,10,12\}$ can be found in~\cite{FK-SciChina,FK-JCTB,FKW,LCM}.
It is also worth noting that Folkman~\cite{Folkman} and Alaeiyan et al.~\cite{AG} showed that there are no cubic semisymmetric graphs of order $2p^2$ or $8p^2$, and it can further be shown that no connected cubic semisymmetric graphs exist of order $4p^2$. Moreover, Lu et al.~\cite{LWX} classified all connected cubic semisymmetric graphs of order $6p^2$. In this paper, we aim to establish a systematic framework for the unified characterization of such graphs.

Before stating the main result, we first introduce some notation.
Let $n$ be an integer. If there exists a unique cubic symmetric (or semisymmetric) graph of order $n$ up to graph isomorphism,
then we use $\F n$ to denote this graph.
If there is more than one cubic symmetric (or semisymmetric) graph of order $n$ up to graph isomorphism, then we use $\F nA$, $\F nB$ etc. to denote the corresponding graphs (see~\cite{CD,CMMP}).

%

\begin{theorem}\label{the:lp2}
Let $\G$ be a connected cubic edge-transitive graph of order $\ell p^2$,
where $\ell$ is a positive integer and $p$ is a prime with $p>\ell$.
Let $A=\Aut(\G)$.
Then either
\begin{enumerate}
  \item[{\rm (1)}] $p=3$ and $\G\cong\F18$; or
  \item[{\rm (2)}] $p\geq5$, $A$ has a normal Sylow $p$-subgroup $P$ satisfying one of the following statement.
  \begin{enumerate}
    \item[{\rm (2.1)}] $P\cong Z_{p^2}$, and either $\G$ is arc-regular and $\G\cong\G_{p,R_1}$, or $\G$ is edge-regular and $\G\cong\G_{p,R_2}$.
    \item[{\rm (2.2)}] $P\cong Z_p^2$, $\ell=2$ and $\G$ is isomorphic to an arc-transitive  graph as listed in {\rm\cite[Theorem 3.5]{FK-AustMath}}.
    \item[{\rm (2.3)}] $P\cong Z_p^2$, $\ell\geq4$, and $\G\cong\Sigma_{p,p,R}$.
    \item[\rm (2.4)] $P\cong Z_p^2$, $\ell\geq4$, and there is a subgroup $M$ of $A$ containing the centralizer $C:=C_A(P)$ such that $\G$ is $M$-edge-regular and the quotient graph $\G_C$ satisfies $(\G_C,M/C)\cong(\F6,Z_3^2)$,  $(\F16,\SL(2,3))$, $(\F48,\SL(2,3)\times Z_3)$, $(\F80,\SL(2,5))$, $(\F240B,\SL(2,5)\times Z_3)$ or $(\S240,\SL(2,5)\times Z_3)$.
  \end{enumerate}
\end{enumerate}
\end{theorem}

\medskip

All graphs referenced in Theorem~\ref{the:lp2} will be formally introduced in Section~\ref{section 3}. Providing a complete and satisfactory characterization of the graph described in Theorem~\ref{the:lp2}(2.4) remains a significant challenge, largely due to the difficulty in precisely determining the structure of the centralizer $C$.
We stress, however, that by imposing the additional assumption that $\ell$ is square-free,
we are able to obtain a much more conclusive and desirable outcome.

\begin{theorem}\label{the:square-free}
Let $\G$ be a connected cubic edge-transitive graph of order $\ell p^2$,
where $\ell$ is a square-free integer and $p$ is a prime with $p>\ell$.
Then exactly one of the following holds:
\begin{enumerate}
  \item [{\rm (1)}]\ $p=3$ and $\G\cong\F18$; or
  \item [{\rm (2)}] $p\geq5$, $\Aut(\G)$ has a normal Sylow $p$-subgroup $P$ and $\G$ satisfies one of the following statements:
  \begin{enumerate}
    \item[{\rm (2.1)}] $P\cong Z_{p^2}$, and either $\G$ is arc-regular and $\G\cong\G_{p,R_1}$, or $\G$ is edge-regular and $\G\cong\G_{p,R_2}$, where $|R_i|$ is an odd square-free with $1\leq i\leq2$;
    \item[{\rm (2.2)}] $P\cong Z_p^2$, $\ell=2$ and $\G$ is isomorphic to an arc-transitive  graph as listed in {\rm\cite[Theorem 3.5(2)]{FK-AustMath}};
    \item[{\rm (2.3)}] $P\cong Z_p^2$, $\ell\geq4$, and $\G\cong\Sigma_{p,p,R}$;
    \item[\rm (2.4)] $P\cong Z_p^2$, $\ell=6$ and $\G$ is isomorphic to one of the graphs listed
    in~{\rm\cite[Theorem 5.3]{FK-2004}} or {\rm \cite[Theorem 4.2]{LWX}}.
  \end{enumerate}
\end{enumerate}
\end{theorem}

\section{Preliminaries}

This section introduces several essential definitions and properties of groups and graphs.

\subsection{Fundamentals of group theory}

Let $m$ and $n$ be two integers, and let $p$ be a prime. Denote by $Z_n$, $D_{2n}$ and $Z_p^n$ a cyclic group of order $n$, a dihedral group of order $2n$ and an elementary abelian $p$-group of order $p^n$, respectively.
For two groups $A$ and $B$, we use $A\times B$ to denote the direct product of $A$ and $B$, and use $A:B$ and $A.B$ to denote the semidirect product and an extension of $A$ by $B$, respectively.

Let $G$ be a group and let $H\leq G$, a subgroup of $G$. Denote by $C_G(H)$ and $N_G(H)$ the {\em centralizer} and {\em normalizer} of $H$ in $G$, respectively. The following well-known result, famously referred to as the ``$N/C$ Theorem", can be found in~\cite[Chapter I, Theorem 4.5]{H1}.

\begin{prop}\label{S}
$C_G(H)\unlhd N_G(H)$ and $N_{G}(H)/C_{G}(H)\lesssim\Aut(H)$.
\end{prop}

The following proposition is related to the maximal subgroups of $\PSL_2(p)$ and $\PGL_2(p)$ with $p$ a prime, coming from~\cite[Theorem 2]{COT} and \cite[P. 392--418]{S}.

\begin{prop}\label{maxPGL2}
Let $p\geq5$ be a prime.
\begin{itemize}
  \item[{\rm(1)}] All maximal subgroups of $\PGL_2(p)$ are isomorphic to $\PSL_2(p)$, $Z_p: Z_{p-1}$, $D_{2(p-1)}$ with $p\neq 5$, $D_{2(p+1)}$, or $S_4$ with $p\equiv\pm3\pmod{8}$;
  \item[{\rm(2)}] All maximal subgroups of $\PSL_2(p)$ are isomorphic to $Z_p: Z_{(p-1)/2}$, $D_{p-1}$ with $p\geq 13$, $D_{p+1}$ with $p\neq 7$, $A_4$ with $p\equiv\pm3\pmod{8}$ and
      $p\equiv\pm3\pmod{10}$, $S_4$ with $p\equiv\pm1\pmod{8}$, or $A_5$ with $p\equiv\pm1\pmod{10}$.
\end{itemize}
\end{prop}

\begin{lem}\label{lem-pgl}
Let $p\geq5$ be a prime.
\begin{itemize}
  \item[{\rm(1)}] Let $H\leq\PGL_2(p)$. If $(\frac{3}{2k}|\PSL_2(p)|)\di|H|$ for $k\in\{2,4,6,8,12,24\}$, then one of the following holds:
\begin{itemize}
  \item[{\rm(1.1)}] $H\cong \PSL_2(p)$ or $\PGL_2(p)$;
  \item[{\rm(1.2)}] $H\cong Z_p:Z_{p-1}$ and $(k,p)=(6,7)$ or $(24,31)$;
  \item[{\rm(1.3)}] $H\cong Z_7:Z_3$ and $(k,p)=(6,7)$.
\end{itemize}
  \item[{\rm(2)}] Let $H\leq\PSL_2(p)$. If $\frac{p(p^2-1)}{8}\di |H|$, then $H\cong \PSL_2(p)$.
\end{itemize}
\end{lem}

\proof (1) We may assume that $H<\PGL_2(p)$. Then there exists a maximal subgroup $M$ of $\PGL_2(p)$ such that $H\leq M$, and so
\begin{equation}\label{eq=1}
(\frac{3}{2k}|\PSL_2(p)|)\di|H|\di |M|.
\end{equation}
By Proposition~\ref{maxPGL2}, $M\cong \PSL_2(p)$, $Z_p:Z_{p-1}$, $D_{2(p\pm1)}$, or $S_4$. If $M\cong Z_p:Z_{p-1}$, then Eq.~\eqref{eq=1} forces that $4k=3(p+1)$ and $|H|=|M|$. Since $k\in\{2,4,6,8,12,24\}$ and $p$ is prime, we have $(k,p)=(6,7)$ or $(24,31)$, and hence the (1.2) holds.
If $M\cong D_{2(p\pm1)}$, then $3p(p\mp1)\di 8k$ by Eq.~\eqref{eq=1}, and since  $k\in\{2,4,6,8,12,24\}$,
we have $p=2$, which is impossible because $p\geq5$.
If $M\cong S_4$, then $p(p^2-1)\di 32k$, and hence $p\leq3$. A contradiction occurs.

Now, let $M\cong\PSL_2(p)$. If $H=M$, then the (1.1) holds.
If $H<\PSL_2(p)$, then Proposition~\ref{maxPGL2} forces that $H$ is a subgroup of $Z_p:Z_{(p-1)/2}$, $D_{p-1}$ with $p\geq 13$, $D_{p+1}$ with $p\neq 7$, $A_4$ with $p\equiv\pm3\pmod{8}$ and $p\equiv\pm3\pmod{10}$, $S_4$ with $p\equiv\pm1\pmod{8}$, or $A_5$ with $p\equiv\pm1\pmod{10}$. Next, we consider each case in turn.
If $H\leq Z_p:Z_{(p-1)/2}$, then $2k=3(p+1)$ because of $(\frac{3}{2k}|\PSL_2(p)|)\di|H|$, and hence $(k,p)=(6,7)$ and $H=Z_7:Z_3$, as shown in (1.3).
If $H\leq D_{p-1}$, then $4k=3p(p+1)$, which is impossible because $p\geq 13$ and $k\in\{2,4,6,8,12,24\}$. If $H\leq D_{p+1}$, then $4k=3p(p-1)$, which is also impossible.
If $H\leq A_4$ with $p\equiv\pm3\pmod{8}$ and $p\equiv\pm3\pmod{10}$ or $H\leq S_4$ with $p\equiv\pm1\pmod{8}$, then $p\geq5$ and $|H|\di 24$. Since $(\frac{3}{2k}|\PSL_2(p)|)\di|H|$, we have $p(p^2-1)\di 32k$, and since $k\in\{2,4,6,8,12,24\}$, we have $p\leq 3$, a contradiction. Similarly, if $H\leq A_5$ with $p\equiv\pm1\pmod{10}$, then $p\geq 11$ and $|H|\di 60$, which also contradicts $(\frac{3}{2k}|\PSL_2(p)|)\di|H|$.

(2) By arguing similarly as (1), the (2) can be easily obtained.
\hfill\qed

The following lemma discusses the non-abelian simple groups that possess a Sylow $p$-subgroup of order $p^2$. The proof of this lemma primarily relies on~\cite[Theorem]{Shen-Zhou}, achieved through a comparison of the order of the non-abelian simple groups. While the proof is straightforward, it can be somewhat intricate.
Consequently, we have included a comprehensive proof in the Appendix for clarity.

\begin{lem}\label{lem-simple-p2}
Let $p$ be an odd prime, and let $T$ be a non-abelian simple group with a Sylow $p$-subgroup $P$ of order $p^2$. If $|T|<192\cdot p^3$, then $(T,P)$ is listed in Table $1$.
\end{lem}

\begin{table}[!htp]
\caption{Simple groups with Sylow subgroups of odd prime square order }\label{table=1}
\begin{center}
\begin{tabular}{|c|c|c|c|c|c|c|c|c|}
\hline
 $T$ &$\A_6$ & $\A_7$& $\PSL_2(8)$& $\PSL_2(9)$&$\PSL_2(17)$& $\PSL_2(19)$&$\PSL_2(25)$&$\PSL_2(49)$\\
\hline
$P$ & $Z_3^2$& $Z_3^2$&$Z_9$&$Z_3^2$&$Z_9$&$Z_9$&$Z_5^2$&$Z_7^2$\\
\hline
\end{tabular}
\end{center}
\end{table}

For a prime $p$ with  $p\di|G|$, the group $G$ is called {\em $p$-solvable} if  $G$ has a normal series
\begin{equation*}
 1=N_s\unlhd N_{s-1}\unlhd\cdots\unlhd N_{0}=G
\end{equation*}
such that $N_{i-1}/N_i$ is a $p$-group or a $p'$-group for all $1\leq i\leq s$.
The group $G$ is {\em solvable} if and only if it is $p$-solvable for every prime $p$ with  $p\di |G|$.

\begin{prop}[{\rm \cite[P. 711, 8.1 Hilfssatz]{H1}}]\label{p-solvable}
Let $G$ be a $p$-solvable group.
Let $V$ be an $n$-dimensional vector space on the finite field $\mathrm{GF}(p)$ with the greatest common divisor $\gcd(n,|G|)=1$.
If $V$ is a faithful and irreducible $G$-module, then $G$ is a cyclic group and $|G|\di (p^n-1)$.
\end{prop}

The lemma below describes the properties of a finite group of order $9p^2$,
which will be used in the proof of the main result.

\begin{lem}\label{lem-centra}
Let $G=P: Q$ be a group of order $9p^2$ with prime $p\geq5$, where $P\cong Z_p^2$ and $Q\cong Z_3^2$. Then $C_G(P)>P$.
\end{lem}

\proof Let $C=C_G(P)$. Clearly, $P\leq C$. Let $K$ be the kernel of $Q$ acting on $P$ by conjugation.
Then $|K|=1, 3$ or $9$, $K\leq Z(G)$ and $P\times K\leq C$. If $K=1$, then $P$ is a minimal normal subgroup of $G$, and so $Q$ is both faithful and irreducible on $P$. Recall that $P\cong Z_p^2$. We may view $P$ as a $2$-dimensional vector space $V$ on the finite field $\mathrm{GF}(p)$.
This implies that $V$ is a faithful and irreducible $Q$-module.
Since $Q$ is a $p$-solvable group  and $\gcd(2,|Q|)=1$, by Proposition~\ref{p-solvable} $Q$ is cyclic, which is a contradiction.
Hence $K\neq1$, and so $C>P$.
\hfill\qed

Let $G$ be a permutation group on a set $\Omega$, and let $\a\in \Omega$. Denote by $G_{\a}$ {\it the stabiliser} of $\a$ in $G$, that is, the largest subgroup of $G$ fixing the point $\a$. The group $G$ is {\it semiregular} on $\Omega$ if $G_{\a}=1$ for every $\a\in\Omega$, and {\it regular} if $G$ is transitive and semiregular on $\Omega$.

\subsection{Fundamentals of graph theory}

For a positive integer $s$, an {\em $s$-arc} of a graph $\G$ is a sequence $(u_0,u_1,\ldots,u_s)$ of $s+1$ vertices such that $u_i$ is adjacent to $u_{i+1}$ for $0\leq i\leq s$ and $u_{j-1}\neq u_{j+1}$ for $1\leq j\leq s-1$.
For a subgroup $G\leq\Aut(\G)$, the graph $\G$ is called {\em $(G,s)$-arc-transitive} if $G$ is transitive on the set of all $s$-arcs of $\G$.
In particular, a graph is said to be {\em $s$-arc-transitive} if it is $(\Aut(\G),s)$-arc-transitive.

Let $\G$ be a $G$-edge-transitive graph and let $N$ be a normal subgroup of $G$. The {\em quotient graph} of $\G$ relative to $N$, denoted by $\G_N$, is the graph with vertices the orbits of $N$ on $V(\G)$ and with two different orbits adjacent if there is at least one edge in $\G$ between the vertices lying in those two orbits.
Further, $\G$ is an {\em $N$-normal cover} (or {\em cover} for short) of $\G_N$ if $\G$ and $\G_N$ have the same valency.
The following two results give the properties of the quotient graphs of cubic edge-transitive graphs.

\begin{prop}[{\rm \cite[Theorem~9]{L2}}] \label{arc}
Let $\G$ be a connected cubic $G$-arc-transitive graph. Then $G$ acts regularly on the $s$-arcs of $\G$ for some positive integer $s$. If $N\unlhd G$ has more than two orbits on $V(\G)$, then $N$ is semiregular on $V(\G)$, $\G_N$ is a cubic
symmetric graph with $G/N\leq\Aut(\G_N)$ acting regularly on the $s$-arcs of $\G_N$.
\end{prop}

\begin{prop}[{\rm \cite[Lemma 3.2]{LWX}}]\label{edge}
Let $\G$ be a connected cubic $G$-edge-transitive graph such that $G$ is intransitive on $V(\G)$. Then $\G$ is a bipartite graph with two partition sets, say $V_0$ and $V_1$. If $N\unlhd G$ is intransitive on each of $V_0$ and $V_1$, then $N$ acts semiregularly on $V(\G)$, $\G_N$ is a cubic graph with $G/N$ as an edge- but not vertex-transitive group of automorphisms.
\end{prop}

Given a group $H$, let $R$, $L$ and $S$ be subsets of $H$ such that $R^{-1} = R$, $L^{-1} = L$ and $R\cup L$ does not contain the identity element of $H$.
The {\em bi-Cayley graph} over $H$ relative to the triple $(R, L, S)$, denoted by $\BiCay(H, R,
L, S)$,
is the graph having vertex set the union of the right part $H_0=\{h_0\mid h\in H\}$ and the left part $H_1=\{h_1\mid h \in H\}$,
and edge set the union of the {\em right edges} $\{\{h_0,g_0\}\mid gh^{-1}\in R\}$, the {\em left edges} $\{\{h_1, g_1\}\mid gh^{-1}\in L\}$ and the {\em spokes} $\{\{h_0,g_1\}\mid gh^{-1}\in S\}$.
The following proposition illustrates the basic properties of a connected bi-Cayley graph over a group $H$, refer to~\cite[Lemma 1.1]{AT} and \cite[Lemma 3.1]{ZF}.

\begin{prop}\label{bi-Cayley}
Let $\G=\BiCay(H,R,L,S)$ be a connected bi-Cayley graph over the group $H$. Then the following hold.
\begin{itemize}
  \item[{\em (1)}] $H$ is generated by $R\cup L\cup S$.
  \item[{\em (2)}] Up to graph isomorphism, $S$ can be chosen to contain the identity of $H$.
  \item[{\em (3)}] For any $\alpha\in\Aut(H)$ and $x\in H$, $\G\cong\BiCay(H, R^\alpha,x^{-1}L^\alpha x,xS^\alpha)\cong\BiCay(H, L^\alpha,x^{-1}R^\alpha x,x(S^{-1})^\alpha)$.
\end{itemize}
\end{prop}

Next, we consider the automorphisms of a bi-Cayley graph $\G=\BiCay(H,R,L,S)$.
For each $g\in H$, define a permutation as follows:
\begin{equation*}
  \mathcal{R}(g):h_i\mapsto (hg)_i,~i\in \{0, 1\}, \forall h\in H.
\end{equation*}
Set $\mathcal{R}(H)=\{\mathcal{R}(g)\mid g\in H\}$.
Then $\mathcal{R}(H)$ is a semiregular subgroup of $\Aut(\G)$ with $H_0$ and $H_1$ as its two orbits.
For an automorphism $\alpha$ of $H$ and $x, y, g \in H$, define two permutations of $V(\G)=H_0\cup H_1$ as follows:
\begin{align*}
 \delta_{\alpha,x,y}&:h_0\mapsto (xh^\alpha)_1,~h_1\mapsto(yh^\alpha)_0,~\forall h\in H;  \\
 \sigma_{\alpha,g}&:h_0\mapsto (h^\alpha)_0,~~~h_1\mapsto(gh^\alpha)_1,~\forall h\in H.
\end{align*}
Set
\begin{eqnarray}
&I=\{\delta_{\alpha,x,y}\mid \alpha\in\Aut(H)\text{~s.t.~} R^\alpha=x^{-1}Lx, L^\alpha=y^{-1}Ry, S^\alpha=y^{-1}S^{-1}x\},\label{Eq-delta}\\
&F=\{\sigma_{\alpha,g}\mid \alpha\in\Aut(H)\text{~s.t.~} R^\alpha=R, L^\alpha=L, S^\alpha=g^{-1}S\}.\label{Eq-sigma}
\end{eqnarray}
An important result related to the automorphisms of a bi-Cayley graph was given by Zhou and Feng in~\cite[Theorem~1.1]{ZF}, as follows.

\begin{prop}\label{bi-Cayley1}
Let $\G=\BiCay(H,R,L,S)$ be a connected bi-Cayley graph over the group $H$. Then the normalizer $N_{{\rm Aut}(\G)}(\mathcal{R}(H))=\mathcal{R}(H): F$ if $I=\phi$ and $N_{{\rm Aut}(\G)}(\mathcal{R}(H))=\mathcal{R}(H)\langle F,\delta_{\a,x,y}\rangle$ if $I\neq\phi$ and $\delta_{\a,x,y}\in I$. Furthermore, for any $\delta_{\a,x,y}\in I$, $\langle \mathcal{R}(H),\delta_{\a,x,y}\rangle$ acts transitively on $V(\G)$.
\end{prop}

If $\mathcal{R}(H)\unlhd \Aut(\G)$, then the bi-Cayley graph $\G=\BiCay(H, R, L, S)$ is said to be a {\em normal bi-cayley graph} over $H$; if $N_{{\rm Aut}(\G)}(\mathcal{R}(H))$ is transitive on the edge-set (resp. arc-set) of $\G$, then $\G$ is said to be a {\em normal edge-transitive} (resp. {\em normal arc-transitive}) {\em bi-Cayley graph}.

\subsection{Some results on cubic edge-transitive graphs}

This paper focuses on cubic edge-transitive graphs. We begin by recalling some known results concerning such graphs. The following proposition is taken from~\cite[Theorem~2]{Folkman}.

\begin{prop}\label{prop=2p2}
All connected cubic edge-transitive graphs of order $2p^2$ with prime $p$ are symmetric.
\end{prop}

These graphs have been classified, and based on~\cite[Theorem~3.2]{FK-AustMath}, we have the following lemma.

\begin{lem}\label{lem2p2}
Let $p$ be an odd prime, and let $\G$ be a connected cubic edge-transitive graph of order $2p^2$.
Then $\Aut(\G)$ has a normal Sylow $p$-subgroup.
\end{lem}
\f{\bf Proof:} Let $P$ be a Sylow $p$-subgroup of $\Aut(\G)$. By~\cite[Theorem~2]{Folkman}, $\G$ is symmetric, and by~\cite[Theorem~3.2]{FK-AustMath}, $P\unlhd A$ when $p\neq 5$ or $7$.
If $p=5$ or $7$, then $\G$ is a symmetric graph of order $50$ or $98$, respectively.
By~\cite{Conder-sym}, there is exactly one symmetric graph of order $50$ and two symmetric graphs of order $98$. These two graphs were given in~\cite{Conder-sym}, and by Magma~\cite{BCP}, each of their full automorphism groups has a normal Sylow $p$-subgroup.\hfill\qed

Cubic symmetric or semisymmetric graphs of order $4p^2$ have been classified in~\cite{FK-JCTB}
and~\cite{HF}, respectively. This classification yields the following result.

\begin{lem}\label{lem4p2}
For an odd prime $p$, there exists no cubic edge-transitive graph of order $4p^2$.
\end{lem}

%

\section{Proof of Theorems~\ref{the:lp2} and~\ref{the:square-free}}\label{section 3}

In this section, we aim to prove Theorems~\ref{the:lp2} and~\ref{the:square-free},
which give a classification of connected cubic edge-transitive graphs of order
$\ell p^2$, where $p$ is a prime and $\ell$ is an even integer satisfying
$2\leq \ell < p$.
In the remainder of this section, we adopt the following assumption.

\medskip
\f{\bf Assumption~I:} \begin{itemize}
  \item $\G$ is a connected cubic edge-transitive graph of order $\ell p^2$ with $v$ being a vertex of $\G$;
  \item $\ell:=2t$ with $t$ being a positive integer;
  \item $p$ is a prime with $p>\ell$, forcing that $p\geq3$;
  \item $A=\Aut(\G)$ and $P$ is a Sylow $p$-subgroup of $A$.
\end{itemize}

\subsection{Proof of Theorem~\ref{the:lp2}}

In this subsection, we aim to prove Theorem~\ref{the:lp2}.
The proof consists of several lemmas. We begin with the first lemma,
which provides the structure and normality of the Sylow $p$-subgroup
$P$ of $\Aut(\G)$.

\begin{lem}\label{lem-P}
Under Assumption~I, if $p>3$, then $P\cong Z_{p^2}$ or $Z_p\times Z_p$, and $P\unlhd A$.
\end{lem}

\f {\bf Proof:} Since $\G$ is edge-transitive, we have that $|A|=\ell p^2|A_v|$ or $tp^2|A_v|$, and $|A_v|\mid 3\cdot2^7$ by~\cite[Propositions 2--5]{DM} and \cite[Theorem A]{Gold}. Since $p>3$, we have $|P|=p^2$, and so $P\cong Z_{p^2}$ or $Z_p\times Z_p$.



We shall show that $P$ is normal in $A$ by using induction on $\ell$.
The lemma holds when $\ell=2$ by Lemma~\ref{lem2p2}. Let $\ell>2$, and let $N$ be a minimal normal subgroup of $A$. Since $\ell<p$ by Assumption~I, we have $p\geq5$. We consider the following three cases depending on $N$.\medskip

\f{\bf Case~1.}\ $p\nmid |N|$\medskip

In this case, $N$ acts semiregularly on $V(\G)$ by Propositions~\ref{arc} and~\ref{edge}. The quotient graph $\G_N$ of $\G$ relative to $N$ is a cubic $A/N$-edge-transitive graph of order $\frac{\ell}{|N|}\cdot p^2$. By induction hypothesis, $PN/N$ is a normal Sylow $p$-subgroup of $A/N$, and so $PN\unlhd A$. Since $\ell<p$ and $|N|\mid \ell$, we have $P\unlhd PN$, forcing that $P$ is characteristic in $PN$. Hence $P\unlhd A$.\medskip

\f{\bf Case~2.}\ $p\mid |N|$ but $p^2\nmid |N|$\medskip

Suppose firstly that $N$ is solvable. Then $N\cong Z_p$.
Consider the quotient graph $\G_N$ of $\G$ relative to $N$.
By Propositions~\ref{arc} and ~\ref{edge}, $\G_N$ is a cubic $A/N$-edge-transitive graph of order $\ell p$. Recall that $\ell>2$ and $p\geq5$. By \cite[Lemma~3.3]{ZF2024}, either $P/N\unlhd A/N$ or
 \[\G_N\cong\F20A, \F20B, \F28, \S110, \F110, \S506, \F506A, \F506B, \F620, ~or~ \F2162.\]
If $P/N\unlhd A/N$, then $P\unlhd A$, as required.

If $\G_N\cong\F20A,\F20B,\F28,\S110$ or $\F110$, then $(\ell,p)=(4,5)$, $(4,7)$ or $(10,11)$ and $|\G|=\ell p^2=100$, $196$ or $1210$. However, it is established in~\cite{Conder-sym,Conder-semi-sym} that no cubic edge-transitive graph of such order exists.

If $\G_N\cong \S506,\F506A,\F620$ or $\F2162$, then $\Aut(\G_N)\cong \PSL_2(p)$ with $p=23$, $31$ or $47$ by Magma~\cite{BCP} (see~\cite{Conder-sym,Conder-semi-sym} for such graphs).
Since the Schur multiplier of $\PSL_2(p)$ is $Z_2$, we have $A=\PSL_2(p)\times Z_p$ by~\cite[Lemma~2.8]{YFD}.
It follows from Propositions~\ref{arc} and~\ref{edge} that the normal subgroup $\PSL_2(p)$ of $A$ is semiregular on $V(\G)$, which is impossible.

If $\G_N\cong \F506B$, then $p=23$ and $\Aut(\G_N)\cong \PGL_2(23)$.
Moreover, $\G_N$ is non-bipartite. It forces that $\Aut(\G_N)$ has an arc-transitive subgroup $\PSL_2(23)$.
Again, $A/N\leq \PSL_2(23)$ by Proposition~\ref{arc}, and the Schur multiplier $\Mult(\PSL_2(23))=Z_2$ by Atlas~\cite{A}, then $A=\PSL_2(23)\times Z_{23}$ by \cite[Lemma~2.8]{YFD}.
By Proposition~\ref{arc}, $\PSL_2(23)$ is semiregular on $V(\G)$, contrary to $Z_3\leq A_v\cap\PSL_2(p)$.

Suppose that $N$ is non-solvable. Then $N$ is a non-abelian simple
group as $p^2\nmid |N|$. By Propositions~\ref{arc} and~\ref{edge}, $N$ is semiregular on $V(\G)$, and hence $|N|\mid \ell p^2$. This implies that $N$ has a normal Sylow-$p$ subgroup, a contradiction.
\medskip

\f{\bf Case~3.}\ $p^2\mid |N|$\medskip

Note that $p\geq 5$. If $N$ is solvable, then $N\cong Z_p\times Z_p$ is a normal Sylow $p$-subgroup of $A$, as required.
Let $N$ be non-solvable. Then $N$ is a non-abelian simple
group or $N\cong T_1\times T_2$, where $T_1\cong T_2$ is non-abelian simple.
If $N$ is semiregular on $V(\G)$, then $|N|\mid \ell p^2$. Since $\ell<p$ and $p^2\mid |N|$, $N$ has a normal Sylow $p$-subgroup of order $p^2$, which contradicts the minimal normality of $N$ in $A$. Hence, $N$ is not semiregular on $V(\G)$.

Assume $N$ is a non-abelian simple group. Recall that $p\geq 5$. We claim that $|N|\geq192\cdot p^3$. Otherwise, Lemma~\ref{lem-simple-p2} implies that $N\cong \PSL_2(25)$ with $p=5$ or $\PSL_2(49)$ with $p=7$. For $N\cong\PSL_2(25)$, since $2<\ell<p$ is an even integer, we have $\ell=4$, yielding that $|\G|=100$.
However, by checking the list of symmetric or semisymmetric graphs given in~\cite{CD} and~\cite{CMMP},
there is no cubic edge-transitive graph of order $100$. For $N=\PSL_2(49)$, we have $\ell=4$ or $6$, and so $|\G|=196$ or $294$,
that is, $\G$ is an edge-transitive graph of order 196 or 294 whose full automorphism group has order at least $|N|=58800$. Again by~\cite{CD} and~\cite{CMMP}, no such graph exists. Hence, $|N|\geq192\cdot p^3$.

If $\G$ is arc-transitive, then $|N_v|\mid 48$, and by Propositions~\ref{arc}, either $N$ is transitive on $V(\G)$ or $N$ has exactly two orbits on $V(\G)$.
For the former case, $192\cdot p^3\leq |N|=|N_v|\cdot |V(\G)|\leq 48\cdot \ell p^2$, forcing that $\ell\geq 4p$, a contradiction.
For the latter case, $192\cdot p^3\leq |N|=|N_v|\cdot \frac{1}{2}|V(\G)|\leq 24\cdot \ell p^2$,
and so $\ell \geq 8p$, a contradiction.
If $\G$ is semisymmetric, then $|N_v|\mid 384$, and by Proposition~\ref{edge},
$N$ has an orbit being one partite set of $\G$. It follows that $|N|=|N_v|\cdot \frac{1}{2}|V(\G)|\leq 192\cdot \ell p^2$. Similarly, since $|N|\geq 192\cdot p^3$, we have $\ell\geq p$, a contradiction.

Finally, assume that $N\cong T_1\times T_2$, where $T_1\cong T_2$ is non-abelian simple.
Since $p^2\mid |N|$ and $p^3\nmid|N|$, there exists an even integer $k$ such that $(k,p)=1$ and $|T_1|=kp$. Moreover, since $T_1$ is simple, we have $\ell<p<k$ and $k$ has at least two distinct prime divisors.

If $\G$ is arc-transitive, then $|N_v|\mid 48$. Since
$N$ has at most two orbits on $V(\G)$, we have $|N|=2^{i}\cdot 3^j\cdot \ell p^2=k^2p^2$ with $-1\leq i\leq 4$ and $0\leq j\leq 1$, and so $k^2=2^i\cdot 3^j\cdot \ell$. In particular, $\ell$ is even, and if $j=1$ then $6\mid \ell$. Since
$2<\ell<p<k$, we have $2<\ell<2^{i}\cdot 3^j$. It can be checked easily that
$(\ell,k)=(6,12)$ or $(12,24)$. Hence, $|T_1|=12p$ or $24p$.
In particular, $T_1$ is a non-abelian simple $\{2,3,p\}$-group.
By~\cite[Theorem I]{HL}, we have $(T_1,\ell,p)=(\A_5,6,5)$ or $(\PSL_2(7),12,7)$, which contradicts $\ell<p$.

If $\G$ is semisymmetric, then $|N_v|\mid 384$ and $N$ has an orbit being one partite set of $\G$. Hence $k^2p^2=|N|=2^i\cdot 3^j\cdot \ell p^2$ with $-1\leq i\leq 7$ and $0\leq j\leq1$,
yielding that $k^2=2^i\cdot 3^j\cdot \ell$. By a similar argument as above,
all possible $(\ell,k)$s are listed in the following:
\begin{itemize}
  \item[(a)] $(\ell,k)\in\{(6,12),(6,24),(6,48),(12,24),(12,48),(18,24),(18,48),(24,48),(24,96),(36,48),(48,96),\\(72,96),(96,192)\}$;
  \item[(b)] $(\ell,k)\in\{(54,72),(54,144),(108, 144),(216, 288)\}$;
  \item[(c)] $(\ell,k)\in\{(50,80),(98,112),(150, 240),(294, 336)\}$.
\end{itemize}

For case (a), we have $|T_1|=2^s\cdot 3\cdot p$ for some integer $s\geq1$ and $p>\ell\geq 6$.
It then follows from~\cite[Theorem I]{HL} that $p= 7$ and $T_1\cong\PSL_2(7)$.
In this case, $\ell=6$ and $\G$ is a cubic semisymmetric graph of order $294$.
By~\cite[P.284]{CMMP}, $\G\cong\S294$ and $|\Aut(\G)|=882<|N|=|T_1|^2$, a contradiction.

For case (b), $|T_1|=2^s\cdot 3^2\cdot p$ for some integer $s\geq1$ and $p>\ell\geq 54$.
However, there is no simple group of such order by~\cite[Theorem I]{HL}.

For case (c), we have $p>50$.
If $k=80$ or $112$, then $T_1$ is a simple group of order $2^4\cdot 5\cdot p$ or $2^4\cdot 7\cdot p$ with $p>50$,
which contradicts~\cite[Theorem I]{HL}.
If $k=240$ or $336$, then $|T_1|=2^4\cdot3\cdot5\cdot p$ or $2^4\cdot3\cdot7\cdot p$, respectively.
From~\cite[Theorem I, Section 3]{HL} we have $T_1\cong\PSL_2(p)$, that is,
$|\PSL_2(p)|=\frac{1}{2}p(p^2-1)=2^4\cdot3\cdot5\cdot p$ or $2^4\cdot3\cdot7\cdot p$.
It forces that $p^2=481$ or $673$, a contradiction.  \hfill\qed

In view of Lemma~\ref{lem-P}, we proceed with the proof of
Theorem~\ref{the:lp2} by distinguishing the following two cases
according as $P$ is cyclic or not.

\begin{lem}\label{lem-C1}
Under Assumption~I, let $C=C_{A}(P)$ and $p>3$. If $P\cong Z_{p^2}$, then $C=P\times R$, where $R$ is the Hall $p'$-subgroup of $C$. Furthermore, the following hold:
\begin{enumerate}
  \item [{\rm (1)}]\ $\G$ is arc-regular or edge-regular;
  \item [{\rm (2)}]\ either $|R|=t=1$ or $|R|=t>2$;
  \item [{\rm (3)}]\ there exist $\lambda\in Z_{p^2}^*$ and $\tau\in \Aut(R)$ such that \[\G\cong\Sigma_{p,R}=\BiCay(C, \emptyset, \emptyset, \{1, ab, a^{1+\lambda}(bb^{\tau})\}),\] where $P=\lg a\rg$, $\lambda^2+\lambda+1=0~ (\mod~ p^2)$ and $R=\lg b, b^\tau\rg$ with $bb^\tau b^{\tau^2}=1$. In particular, $\G$ is a normal bi-Cayley graph over $C$.
\end{enumerate}
\end{lem}
\f{\bf Proof:}  Since $P$ is abelian, we have $P\leq C$, and since $P$ is a Sylow $p$-subgroup of $A$, we have $C=P\times R$, where $R$ is the Hall $p'$-subgroup of $C$. By Lemma~\ref{lem-P}, $P\unlhd A$, yielding that $C\unlhd A$ and $A/C\lesssim\Aut(P)\cong Z_{p(p-1)}$ by Proposition~\ref{S}. Since $R$ is characteristic in $C$ and $C\unlhd A$, we have $R\unlhd A$. Clearly, $R$ has more than $p^2$ orbits on $V(\G)$.
By Propositions~\ref{arc} and~\ref{edge}, $R$ is semiregular on $V(\G)$. It follows that $C=P\times R$ is also semiregular on $V(\G)$.
In particular, for each $v\in V(\G)$, $A_v\cap C=1$ and so
$$A_v\cong A_v/(A_v\cap C)\cong A_vC/C\leq A/C\leq\Aut(P)\lesssim Z_{p(p-1)},$$
which implies that $A_v\cong Z_3$.
Thus, $\G$ is either arc-regular or edge-regular, (1) holds.

Since $C$ is semiregular on $V(\G)$, we have $|C|\mid 2tp^2$, and so $|R|\mid 2t$. If $C$ is transitive on $V(\G)$, then $|R|=2t$ and the quotient graph $\G_R$ would be a cubic graph of odd order $p^2$, which is impossible. If $C$ has at least three orbits on $V(\G)$,
then Propositions~\ref{arc} and~\ref{edge} imply that the quotient graph $\G_C$ is a cubic $A/C$-edge-transitive graph. Since $A/C$ is cyclic and $\G_C$ is connected, $\G_C\cong K_{3,3}$.
However, $K_{3,3}$ has no cyclic edge-transitive group of automorphisms. Therefore, $C$ has exactly two orbits on $V(\G)$, and hence $|R|=t$. Since there exists no cubic edge-transitive graph of order $4p^2$ when $p>3$ by Lemma~\ref{lem4p2}, we have $t\neq 2$, and so either $|R|=t=1$ or $|R|=t>2$, (2) holds.

Since $C$ is a semiregular normal subgroup of $A$ with exactly two orbits on $V(\G)$,
$\G$ is a normal bi-Cayley graph over $C$. Since $\G$ is edge-transitive, each orbit of $C$ forms an independent subset of $\G$. By Proposition~\ref{bi-Cayley}~(2), we may let
\[\G=\Sigma_{p,R}=\BiCay(C, \emptyset, \emptyset, \{1, x, y\}),\]
where $C=\lg x, y \rg$ because of the connectedness of $\G$.
Since $\G$ is edge-transitive, $A_{1_0}$ has an element $\sigma$ of order 3 permuting $\{1_1,x_1,y_1\}$ cyclically,
and since $C\unlhd A$, Proposition~\ref{bi-Cayley1} implies that $\sigma=\sigma_{\alpha,g}$ for some $\alpha\in\Aut(C)$ and $g\in C$, that is,
$$(1_1,x_1,y_1)^{\sigma_{\a,g}}=((g1^{\alpha})_1,(gx^{\alpha})_1,(gy^{\alpha})_1)=(x_1,y_1,1_1).$$
It yields that $g=x$, $y=xx^\alpha$ and $xx^\alpha x^{\alpha^2}=1$.
Recall that $C=P\times R$, where $P\cong Z_{p^2}$ and $R$ are characteristic subgroups of $C$. Since $C=\lg x, y\rg$, we may let $x=ab$ with $\langle a\rangle=P$ and $1\neq b\in R$.
It follows from $xx^\alpha x^{\alpha^2}=1$ that $aa^\alpha a^{\alpha^2}=1$ and $bb^\alpha b^{\alpha^2}=1$.
Assume that $a^\alpha=a^\lambda$ for some $\lambda\in Z_{p^2}^*$.
Then $aa^\alpha a^{\alpha^2}=1$ gives that $1+\lambda+\lambda^2=0~ (\mod~ p^2)$.
Since $p>3$, we have that $\lambda$ is an element in $Z_{p^2}^*$ of order 3,
forcing that $p\equiv1 (\mod~3)$.
Let $\tau$ be the restriction of $\alpha$ on $R$.
Then $\tau\in \Aut(R)$, and $bb^\alpha b^{\alpha^2}=1$ implies that $bb^{\tau} b^{{\tau}^2}=1$.
Since $P\times R=C=\lg x, y\rg=\lg x, x^\alpha\rg=\lg ab, a^\lambda b^\tau\rg$, it follows that $R=\lg b, b^\tau\rg$. This completes the proof.
\hfill\qed

Lemma~\ref{lem-C1} presents a family of cubic graphs. In the following lemma, we analyze their symmetry properties.

\begin{lem}\label{lem-normal-case-graph-construction}
Let $p$ be a prime with $p>3$ and $p\equiv 1(\mod~3)$. The graph $\Sigma_{p, R}$ is either arc-regular or edge-regular, and furthermore, it is arc-regular if and only if there exists an automorphism $\b\in\Aut(R)$ such that $(b^\b, (bb^\tau)^\b)=(b^{-1}, (bb^\tau)^{-1})$.
\end{lem}

\f{\bf Proof:}
It is easy to see that $\lg ab, a^{1+\ld}(bb^\tau)\rg=P\times R=C$ and so $\G$ is connected.
Since $R=\lg b, b^\tau\rg$ with $bb^\tau b^{\tau^2}=1$, one has $b^{\tau^2}=(bb^\tau)^{-1}$ and hence $b^{\tau^3}=(b^\tau b^{\tau^2})^{-1}=b$. This implies that $\tau$ has order $1$ or $3$.
Note that $\ld^2+\ld+1\equiv0\ (\mod\ p^2)$.
Let $\a$ be the map on $C$ induced by $a\mapsto a^\lambda$ and $b\mapsto b^\tau$. Then $\a\in\Aut(C)$.
A direct computation  shows that \[(ab)(ab)^\a(ab)^{\a^2}=(ab)(a^\ld b^\tau)(a^{\ld^2} b^{\tau^2})=a^{1+\ld+\ld^2}(bb^\tau b^{\tau^2})=1.\]
It yields that $(ab)^{-1}=a^{\lambda+\lambda^2}b^{\tau}b^{\tau^2}$, and so \[\{1,ab,a^{1+\lambda}(bb^\tau)\}^{\a}
 =\{1,a^{\lambda}b^{\tau},a^{\lambda+\lambda^2}b^{\tau}b^{\tau^2}\}
 =(ab)^{-1}\{1,ab,a^{1+\lambda}(bb^\tau)\}.\]
By Eq.~\eqref{Eq-sigma}, $\sigma_{\a,ab}\in \Aut(\Sigma_{p, R})$. In particular, $1_0^{\sigma_{\a,ab}}=1_0$ and $\sigma_{\a,ab}$ fixes $1_0$ and permutes the neighbors  $1_1,(ab)_1,(a^{1+\lambda}(bb^\tau))_1$ of $1_0$  cyclically.
It follows from Proposition~\ref{bi-Cayley1} that $\Sigma_{p,R}$ is a connected cubic edge-transitive bi-Cayley graph over $C$. By Lemma~\ref{lem-C1}, $\Sigma_{p,R}$ is either edge-regular or arc-regular.

Assume that there exists a $\b\in\Aut(R)$ such that $(b^\b, (bb^\tau)^\b)=(b^{-1}, (bb^\tau)^{-1})$.
Then the map $\phi: a\mapsto a^{-1},~b\mapsto b^{\b}$ induces an automorphism of $C$, and
\[\{1, ab, a^{1+\ld}(bb^\tau)\}^\phi=\{1, ab, a^{1+\ld}(bb^\tau)\}^{-1}.\]
By Eq.~\eqref{Eq-delta}, $\delta_{\phi,1,1}\in \Aut(\Sigma_{p,R})$, and by Proposition~\ref{bi-Cayley1}, $\Sigma_{p,R}$ is vertex-transitive. Therefore, $\Sigma_{p,R}$ is arc-transitive, and so it is arc-regular.

Conversely, assume that $\Sigma_{p,R}$ is arc-regular. For convenience of the statement, we identify $C$ with $\lg \R(a)\rg\times \lg\R(b), \R(b^\a)\rg$. By Lemma~\ref{lem-C1},
$\Sigma_{p,R}$ is a normal bi-Cayley graph over $C$. Since $\Sigma_{p,R}$ is arc-regular, by Proposition~\ref{bi-Cayley1}, there exists $\pi\in\Aut(C)$ such that $\d_{\pi, 1, 1}$ swaps $1_0$ and $1_1$. Clearly, $\d_{\pi, 1, 1}^2$ fixes the arc $(1_0, 1_1)$, so $\d_{\pi, 1, 1}^2=1$. By Eq.~\eqref{Eq-delta}, $\pi$ is an automorphism of $C$ of order $1$ or $2$, and
\[\{1, ab, a^{1+\ld}(bb^\tau)\}^\pi= \{1, ab, a^{1+\ld}(bb^\tau)\}^{-1}.\]
It implies that
 \[(ab, a^{1+\ld}(bb^\tau))^\pi=((ab)^{-1}, (a^{1+\ld}(bb^\tau))^{-1})\ {\rm or}\ ((a^{1+\ld}(bb^\tau))^{-1}, (ab)^{-1}).\]
Recall that $C=P\times R$ and both $P$ and $R$ are characteristic subgroups of $C$. It follows that either
\[(a, a^{1+\ld})^\pi=(a^{-1}, a^{-1-\ld})\ {\rm and}\ (b, bb^\tau)^\pi=(b^{-1}, (bb^\tau)^{-1}),\]
or
\[(a, a^{1+\ld})^\pi=(a^{-1-\ld}, a^{-1})\ {\rm and}\ (b, bb^\tau)^\pi=((bb^\tau)^{-1}, b^{-1}).\]
Suppose that the latter holds, that is, $a^\pi=a^{-1-\ld}$. Since $\pi$ has order $1$ or $2$ and $P=\lg a\rg\cong Z_{p^2}$, we have $a^\pi=a^{-1-\ld}=a$ or $a^{-1}$. It follows that $-1-\ld\equiv 1\ (\mod\ p^2)$ or $-1-\ld\equiv-1\ (\mod\ p^2)$. If $-1-\ld\equiv 1\ (\mod\ p^2)$, then $\ld\equiv-2\ (\mod\ p^2)$, and since $\ld^2+\ld+1\equiv0\ (\mod\ p^2)$, we have $p^2\mid 3$, a contradiction. If $-1-\ld\equiv-1\ (\mod\ p^2)$, then $p^2\mid \ld$, which is contradict to $\ld\in\mz_{p^2}^*$. Thus, $(a, a^{1+\ld})^\pi=(a^{-1}, a^{-1-\ld})$ and $(b, bb^\tau)^\pi=(b^{-1}, (bb^\tau)^{-1})$. Noting that $R$ is characteristic in $C$,
let $\b=\pi_{|R}$ be the restriction of $\pi$ on $R$. Then $\b\in\Aut(R)$ and
$(b, bb^\tau)^\b=(b^{-1}, (bb^\tau)^{-1})$. The proof is complete.\hfill\qed

If $R$ has an automorphism $\b$ such that $(b^\b, (bb^\tau)^\b)=(b^{-1}, (bb^\tau)^{-1})$, then we denote the graph $\Sigma_{p,R}$ by $\G_{p,R_1}$; otherwise, we denote it by $\G_{p,R_2}$. Under Assumption~I, if $p=3$, then $\ell=2$, and so $\G$ is a cubic edge-transitive graph of order $18$. By~\cite{Conder-sym,Conder-semi-sym}, $\G\cong \F18$ and $|\Aut(\F18)|=2^3\cdot3^3$, implying that each Sylow $3$-subgroup of $\Aut(\F18)$ is not isomorphic to $Z_9$. Therefore, combining with Lemmas~\ref{lem-P}--\ref{lem-normal-case-graph-construction}, we have the following theorem.

\begin{theorem}\label{the:cycle}
Under Assumption~I, if $\Aut(\G)$ contains a Sylow $p$-subgroup $P\cong Z_{p^2}$ with $p>3$, then $p\equiv1~(\mod~3)$ and either $\G\cong\G_{p,R_1}$ is arc-regular, or $\G\cong\G_{p,R_2}$ is edge-regular.
\end{theorem}

We now turn to the case where $P$ is not cyclic, that is,
$P\cong Z_p\times Z_p$. Before proceeding with the discussion
of this case, we first establish a lemma that investigates
the structural properties of cubic $G$-edge-transitive graphs,
where $G\leq \GL(2,p)$ and $p$ is an odd prime. This lemma
will be used later in the treatment of the case $P\cong Z_p\times Z_p$.

\begin{lem}\label{lem-GL}
Let $\G$ be a connected cubic $G$-edge-transitive graph, where $G\leq \Aut(\G)$. If $G\leq\GL_2(p)$ with $p$ an odd prime and $p\nmid |G|$, then $G$ contains a subgroup $M$ such that $M$ is regular on $E(\G)$ and one of the following holds:
\begin{enumerate}
  \item [{\rm (1)}]\ $M=(G\cap \SL_2(p))\times Z$, where $G\cap \SL_2(p)\cong Z_3, \SL_2(3),\SL_2(5)$ or $2 ^\centerdot S_4^{-1}$, $Z\leq Z(\GL_2(p))$ and $Z\cong Z_3$;
  \item [{\rm (2)}]\ $M=G\cap\SL_2(p)\cong \SL_2(3)$ or $\SL_2(5)$.
\end{enumerate}

\end{lem}

\f {\bf Proof:} Let $H=G\cap\SL_2(p)$. Then $H\unlhd G$ and $G/H\leq Z_{p-1}$.
According to \cite[Chapter~3, Theorem~6.17]{S} or \cite[Tables~8.1--8.2]{Maximal-subgroups}, $H$ is isomorphic to one of the following groups:
\begin{enumerate}
  \item [{\rm (i)}]\ a cyclic group;
  \item [{\rm (ii)}]\ the group $\lg x, y\ |\ x^{2n}=1, x^n=y^2, y^{-1}xy=x^{-1}\rg$ with $n>1$;
  \item [{\rm (iii)}]\ $\SL_2(3)$, $\SL_2(5)$, or $2^\centerdot S_4^{-1}$.
\end{enumerate}

Let $P$ be a Sylow $3$-subgroup of $G$ and let $M=HP$.
Since $G/H\leq Z_{p-1}$, one has $M\unlhd G$, and so all Sylow 3-subgroups of $G$ are contained in $M$.
Since $\G$ is $G$-edge-transitive, we have $3\di |G_u|$ for all $u\in V(\G)$ (see~\cite{Tutte1947,Tutte1959}).
For an edge $\{u,v\}$ of $\G$, let $Q_u$ and $Q_v$ be Sylow $3$-subgroups of $G_u$ and $G_v$, respectively. Then $1\neq Q_u\leq M_u$ and $1\neq Q_v\leq M_v$.
Moreover, $Q_u$ is transitive on $\G(u)$ and $Q_v$ is transitive on $\G(v)$.
It follows that $\lg M_u, M_v\rg$ is edge-transitive on $\G$, and consequently, $M$ is edge-transitive on $\G$.

Note that $\G$ has valency $3$ and $M_{uv}$ is a $2$-subgroup. Since $\G$ is connected, no non-trivial subgroup of $M_{uv}$ is normal in $M$,
and since $M=HP$ with $H\unlhd M$ and $P$ a Sylow $3$-subgroup of $M$, we have $M_{uv}\leq H$.
Moreover, since $H\leq\SL_2(p)$ and $\SL_2(p)$ has an unique involution, it forces that $M_{uv}=1$, and since $|M_u:M_{uv}|=|M_v:M_{uv}|=3$, we have $M_u\cong M_v\cong Z_3$. In particular, $Q_u=M_u$ and $Q_v=M_v$.

If $M$ is vertex-transitive on $\G$, then $M$ acts regularly on the arc-set of $\G$.
Take two distinct neighbours $u, w$ of $v$ in $\G$.
The arc-transitivity of $M$ on $\G$ implies that there exist $\a,\b\in M$ such that $(v,u)^\a=(u,v)$ and $(v,w)^\b=(w,v)$. Clearly, $\a\neq \b$.
Then $\a^2\in M_{vu}$ and $\b^2\in M_{vw}$.
Since $M_{uv}=1$, both $\a$ and $\b$ are involutions, which is impossible because $M$ contains at most one involution. Therefore, $M$ is not vertex-transitive on $\G$, and so $M$ acts regularly on the edge-set of $\G$. In particular, $\G$ is bipartite.

Recall that $H=G\cap \SL_2(p)=M\cap \SL_2(p)$.
Since $H\leq M$, we have $H_u\leq M_u$, yielding $H_u=1$ or $M_u$.
In what follows, we will show that part (1) holds when $H_u=1$, and part (2) holds when $H_u=M_u$.

Assume first that $H_u=M_u$.
Then Proposition~\ref{edge} implies that $H$ is transitive on the partite sets of $\G$. Moreover, since $M$ is not vertex-transitive on $\G$, we have $H=M$. Since $M_u$ and $M_v$ are two different subgroups of order $3$ in $M=H$, $H$ is not isomorphic to one of the groups in (i) and (ii), and thus $H\cong \SL_2(3)$, $\SL_2(5)$, or $2^\centerdot S_4^{-1}$.
If $H\cong 2 ^\centerdot S_4^{-1}$, then $\G$ has order $32$.
By~\cite[Table 1]{CD} and~\cite[Section 7]{CMMP}, there is exactly one cubic edge-transitive graph of order $32$.
However, Magma~\cite{BCP} computation shows that the automorphism group of such a graph has no subgroup isomorphic to $2 ^\centerdot S_4^{-1}$. Thus $H\cong \SL_2(3)$ or $\SL_2(5)$, as in part (2).

Assume now that $H_u=1$.
Then  $H<M$ and $M_u\nleq\SL_2(p)$ as $H=G\cap \SL_2(p)=M\cap \SL_2(p)$.
Recall that $M_u\cong M_v\cong Z_3$.
Since $M/H\leq G/H\leq Z_{p-1}$ is cyclic and
$M=\lg M_u,M_v\rg$, the quotient $M/H$ is generated by the images of
$M_u$ and $M_v$, each of order at most $3$.
This implies that $M/H\cong Z_3$.

Let $Z$ be the center of $\GL_2(p)$.
Then $M_u\cap Z= 1$ or $M_u$.
If $M_u\cap Z= M_u$, then $M_u\leq Z$, and so $M_u\lhd M$.
This implies that $M_u$ fixes $\G(u)$ pointwise, and by the connectivity of $\G$,
we have $M_u$ fixes $V(\G)$ pointwise, which is impossible.
Therefore, $M_u\cap Z= 1$, and so $M_uZ/Z\cong Z_3$.
It follows from
$|\GL_2(p)/Z:\SL_2(p)Z/Z|=2$ that $M_uZ/Z\leq \SL_2(p)Z/Z$, that is,  $M_uZ\leq\SL_2(p)Z$. Let $M_u=\lg g\rg$.
Since $M_u\cap Z=M_u\cap\SL_2(p)=1$,
we have $g=hz$ for some $1\neq h\in\SL_2(p)$ and $1\neq z\in Z$, and since $g^3=1$,
we have $h^3=z^{-3}\in Z(\SL_2(p))\cong Z_2$.
In particular, $h^2$ and $z^2$ has order $3$.
Since $H\unlhd M$, we have $H^g=H^h=H$, yielding that $h\in N_{{\rm SL}_2(p)}(H)$.
Let $\tau$ be the automorphism of $H$ induced by conjugation by $g$, equivalently by $h$.

Let $\tau=1$. Then $g$ centralizes $H$.
Since $g\in M_u$ and $H_u=1$, we have $g\notin H$.
It follows from $M/H\cong Z_3$ that $M=\langle H,g\rangle=H\times \langle g\rangle$.
Let $M_v=\lg x\rg$. Then $x=yg^i$ for some
$y\in H$ and $i\in\{0,1,2\}$.
Since $x^3=1$ and $g$ centralizes $H$,
we conclude that $y^3=1$ and
\[
M=\lg M_u,M_v\rg=\lg g,x\rg=\lg g,y\rg .
\]
As $M=H\lg g\rg$ and $H\cap\lg g\rg=1$, we obtain
$H=\lg y\rg$. Since $y\neq1$, we have $H\cong Z_3$ and $M\cong Z_3^2$.
Recall that $p$ is an odd prime.
If $p=3$, then $3\nmid |G|$ by
hypothesis, contrary to $Z_3^2\cong M\leq G$.
Therefore, $p\neq3$.
Now, the element $y$ is a non-central semisimple element (namely, the order of $y$ is not equal to $p$).
According to \cite[Tables~8.1--8.2]{Maximal-subgroups},
$C_{\SL_2(p)}(H)$ is a subgroup of $Z_{p-1}$ or $Z_{p+1}$.
In particular, $C_{\SL_2(p)}(H)$ has a unique subgroup of order $3$, namely $H$.
On the other hand, from $g=hz$ we get $h\in C_{\SL_2(p)}(H)$.
Since $h^2$ has order $3$, we obtain $h^2\in H$,
and so $z^2=h^{-2}g^2\in M$.
Thus, $Z_3\cong\langle z^2\rangle \leq Z$,
and $M=H\times \langle z^2\rangle$.
Thus conclusion (1) holds with $H\cong Z_3$.

Let $\tau\neq 1$.
Then $\tau$ has order $3$.
Note that $H$ is isomorphic to one of the groups listed in (i)--(iii) as above.
If $H$ is isomorphic to the group in (i), namely, $H$ is cyclic,
then $N_{\SL_2(p)}(H)/C_{\SL_2(p)}(H)$ has order at most $2$ (see \cite[Tables~8.1--8.2]{Maximal-subgroups}),
so no non-trivial automorphism of order $3$ occurs,
a contradiction.
If $H$ is isomorphic to the group in (ii), then we may write
\[
H\cong \lg a,b\mid a^{2n}=1,\ a^n=b^2,\ b^{-1}ab=a^{-1}\rg  \text{~with~}n\geq2.
\]
In this case $\langle a\rangle$ is characteristic in $H$,
and hence $h$ normalizes $\langle a\rangle$.
Again by \cite[Tables~8.1--8.2]{Maximal-subgroups},
we conclude that $H\leq Q_{2(p+1)}$ and $N_{\SL(2,p)}(\langle a\rangle)\leq D_{2(p+1)}$.
Thus, the order of $h^2$ divides $2(p+1)$ and so $3\bigm| \frac{p+1}{2}$.
On the one hand, $z^2\in Z\cong Z_{p-1}$, we have $3\bigm|\frac{p-1}{2}$,
which is impossible as $(\frac{p+1}{2},\frac{p-1}{2})=1$.
Finally, suppose that $H$ is isomorphic to one of the groups in (iii),
namely, $H\cong \SL_2(3)$, $\SL_2(5)$, or $2^\centerdot S_4^{-1}$.
By \cite[Tables~8.1--8.2]{Maximal-subgroups}, we have $N_{\SL_2(p)}(H)=H$,
and so $h\in H\leq M$ as $h\in N_{\SL_2(p)}(H)$.
It follows from $g=hz\in M_u$ that $z=h^{-1}g\in M$, and hence $z^2\in M\setminus H$.
In particular, $\langle z^2\rangle\cong Z_3$ is a normal subgroup of $M$ and $\langle z^2\rangle \cap H=1$.
Since $M/H\cong Z_3$, we conclude that $M=H\times \langle z^2\rangle$.
This gives part (1) and completes the proof. \hfill\qed

We are now ready to proceed with the discussion of the case
$P\cong Z_p\times Z_p$.

\begin{lem}\label{lem-C2}
Under Assumption~I, let $p>3$ and $C=C_A(P)$.
If $P\cong Z_p\times Z_p$, then $C=P\times R$ with $R$ the Hall $p'$-subgroup of $C$, and one of the following holds:
\begin{enumerate}
  \item [{\rm (1)}] $\ell=2$ and $\G$ is arc-transitive which is listed in {\rm\cite[Theorem 3.5]{FK-AustMath}};
  \item [\rm (2)] $\ell\geq4$, $\G$ is a normal bi-Cayley graph over $Z_p\times Z_p\times R$ and $\G\cong\Sigma_{p,p,R}=\BiCay(Z_p\times Z_p\times R,\emptyset,\emptyset,\{1,ac,bcc^\tau\})$, where
      $Z_p\times Z_p=\lg a,b\rg$, $R=\lg c,c^\tau\rg$ is a group of order $t$, $\tau\in \Aut(R)$ with $cc^{\tau}c^{\tau^2}=1$;
  \item [{\rm (3)}] $\ell\geq4$, $\G$ is a cover of $\G_C$, and there is a subgroup $M$ of $A$ such that $\G_C$ is $M/C$-edge-regular and $(\G_C,M/C)=(\F6,Z_3^2)$,  $(\F16,\SL_2(3))$, $(\F48,\SL_2(3)\times Z_3)$, $(\F80,\SL_2(5))$, $(\F240B,\SL_2(5)\times Z_3)$ or $(\S240,\SL_2(5)\times Z_3)$.
      Moreover, $\G$ is $M$-edge-regular.
\end{enumerate}
\end{lem}

\f{\bf Proof:} Let $C=C_{A}(P)$. Since $P$ is a Sylow $p$-subgroup of $A$ that is abelian, we have $C=P\times R$ where $R$ is the Hall $p'$-subgroup of $C$.
By Lemma~\ref{lem-P}, $C\unlhd A$, and since $R$ is characteristic in $C$, we have $R\unlhd A$.

Recall that $\ell$ is even.
If $\ell=2$, then $|V(\G)|=2p^2$, and Proposition~\ref{prop=2p2} implies that $\G$ is arc-transitive, and these graphs have been classified completely in~\cite[Theorem 3.5]{FK-AustMath}.

Assume $\ell\geq4$.
Notice that $P$ (or $R$) acting on $V(\G)$ has at least $\ell$ (or $p^2$) orbits.
By Propositions~\ref{arc} and~\ref{edge}, both $P$ and $R$ are semiregular on $V(\G)$.
We conclude that $C$ is semiregular on $V(\G)$.
If $C$ is transitive on $V(\G)$, then $|C|=\ell p^2$ and so $|R|=\ell$.
It follows from Propositions~\ref{arc} and~\ref{edge} that $\G_R$ is a connected cubic graph of odd order $p^2$, which is impossible.
If $C$ has exactly two orbits on $V(\G)$, then $\G$ is a normal bi-Cayley graph over $C$,
and by Proposition~\ref{bi-Cayley}~(2), we may let
$\G=\Sigma_{p,p,R}=\BiCay(C, \emptyset, \emptyset, \{1, x, y\})$, where $C=\lg x, y \rg\cong Z_p\times Z_p\times R$ because of the connectedness of $\G$. Since $\G$ is edge-transitive,
$A_{1_0}$ has an element $\sigma$ of order 3 permuting $\{1_1,x_1,y_1\}$ cyclically,
and since $C\unlhd A$, Proposition~\ref{bi-Cayley1} implies that $\sigma=\sigma_{\alpha,g}$ for some $\alpha\in\Aut(C)$ and $g\in C$, that is,
$$(1_1,x_1,y_1)^{\sigma_{\a,g}}=((g1^{\alpha})_1,(gx^{\alpha})_1,(gy^{\alpha})_1)=(x_1,y_1,1_1).$$
It yields that $g=x$, $y=xx^\alpha$ and $xx^\alpha x^{\alpha^2}=1$.
Recall that $C=P\times R$, where $P\cong Z_{p}\times Z_p$ and $R$ are characteristic subgroups of $C$. Since $C=\lg x, y\rg$, we may let $x=ac$ and $y=bd$ with $\langle a,b\rangle=P$ and $\lg c,d\rg=R$.
It follows from $y=xx^\alpha$ and $xx^\alpha x^{\alpha^2}=1$ that
$$aa^{\a}=b,~cc^{\a}=d,~aa^\alpha a^{\alpha^2}=1,~cc^\alpha c^{\alpha^2}=1.$$
Consequently, we have $a^{\a}=a^{-1}b$ and $aa^\alpha a^{\alpha^2}=ba^{\a^2}=b(a^{-1}b)^{\a}=bb^{-1}ab^{\a}=1$, that is, $b^{\a}=a^{-1}$.
Similarly, $c^{\a}=c^{-1}d$ and $d^{\a}=c^{-1}$.
Let $\tau$ be the restriction of $\alpha$ on $R$.
Then $\tau\in \Aut(R)$, and $cc^{\tau} c^{{\tau}^2}=1$.
Therefore, $\G=\BiCay(C, \emptyset, \emptyset, \{1, ac, bcc^{\tau}\})$ with $cc^{\tau} c^{{\tau}^2}=1$ and $R=\lg c,c^{\tau}\rg$.

Finally, assume $C$ has at least three orbits on $V(\G)$. Then $\G_C$ is a cubic $A/C$-edge-transitive graph. In particular, $A/C$ is a $p'$-group, and $A/C\lesssim\Aut(P)\cong\GL_2(p)$ by Proposition~\ref{S}.
By Lemma~\ref{lem-GL}, $A/C$ contains a subgroup $M/C$ which is regular on $E(\G_C)$, and
\begin{align*}
M/C\cong Z_3\times Z_3, \SL_2(3)\times Z_3, \SL_2(5)\times Z_3, 2^{.} S_4^{-1}\times Z_3, \SL_2(3) \text{~or~} \SL_2(5).
\end{align*}
Since $M/C$ is regular on $E(\G_C)$, we have  $|V(\G_C)|=\frac{2}{3}\cdot|E(\G_C)|=\frac{2}{3}\cdot|M/C|=6,48,240,96,16$ or $80$, respectively.
By~\cite{CD} and~\cite{CMMP}, we have
$\G_C\cong \F6$, $\F48$, $\F240A$, $\F240B$, $\F240C$, $\S240$, $\F96A$, $\F96B$, $\F16$ or $\F80$.
However, it can be checked by Magma~\cite{BCP} that
neither $\Aut(\F240A)$ nor $\Aut(\F240C)$ has a subgroup isomorphic to $\SL_2(5)\times Z_3$.
Similarly, neither $\Aut(\F96A)$ nor $\Aut(\F96B)$ has a subgroup isomorphic to $2^{.} S_4^{-1}\times Z_3$.
Thus, $\G_C\ncong\F240A,\F240C,\F96A$ or $\F96B$.
Moreover, through checking the subgroups of automorphisms of other graphs by Magma~\cite{BCP}, we have that:
\begin{itemize}
  \item[(i)] $\Aut(\G_C)=M/C:D_8$, and $(\G_C,M/C)\cong(\F6,Z_3^2)$ or $(\F80,\SL_2(5))$;
  \item[(ii)] $\Aut(\G_C)=M/C:Z_2^2$, and $(\G_C,M/C)\cong(\F16,\SL_2(3))$, $(\F48,\SL_2(3)\times Z_3)$, $(\F240B,\SL_2(5)\times Z_3)$ or $(\S240,\SL_2(5)\times Z_3)$.
\end{itemize}
Since $\G_C$ is $M/C$-edge-regular, one has $\G$ is $M$-edge-regular.
This completes the proof.\hfill\qed

\begin{lem}\label{lem-C2-normal-case-graph}
Let $p$ be a prime with $p>3$. The graph $\Sigma_{p, p, R}$ is a connected cubic edge-transitive graph.
\end{lem}

\f{\bf Proof:}
It is easy to see that $\lg ac, b(cc^\tau)\rg=P\times R=C$ and so $\G$ is connected. Since $R=\lg c, c^\tau\rg$ with $cc^\tau c^{\tau^2}=1$, one has $c^{\tau^2}=(cc^\tau)^{-1}$ and hence $c^{\tau^3}=(c^\tau c^{\tau^2})^{-1}=c$. This implies that $\tau$ has order $1$ or $3$. Let $\a$ be the map on $C$ induced by $a\mapsto a^{-1}b$, $b\mapsto a^{-1}$ and $c\mapsto c^\tau$. Then $\a\in\Aut(C)$.
A direct computation shows that \[(ac)(ac)^\a(ac)^{\a^2}=(a\cdot a^{-1}b\cdot b^{-1})(cc^\tau c^{\tau^2})=1.\]
It yields that $(ac)^{-1}=a^{-1}c^{\tau}c^{\tau^2}$, and so \[\{1,ac,b(cc^\tau)\}^{\a}
 =\{1,a^{-1}bc^{\tau},a^{-1}c^{\tau}c^{\tau^2}\}
 =(ac)^{-1}\{1,ac,b(cc^\tau)\}.\]
By Eq.~\eqref{Eq-sigma}, $\sigma_{\a,ac}\in \Aut(\Sigma_{p,p, R})$.
In particular, $1_0^{\sigma_{\a,ac}}=1_0$ and $\sigma_{\a,ac}$ fixes $1_0$ and permutes the neighbors  $1_1,(ac)_1,(b(cc^\tau))_1$ of $1_0$ cyclically. It follows from Proposition~\ref{bi-Cayley1} that $\Sigma_{p,p,R}$ is a connected cubic edge-transitive bi-Cayley graph over $C$.
\hfill\qed

Combining Lemma~\ref{lem-P}--\ref{lem-C1}, Theorem~\ref{the:cycle} and Lemma~\ref{lem-C2}, we can obtain the proof of Theorem~\ref{the:lp2}. It should be noted that the main purpose of this paper is to give a structural characterization of cubic edge-transitive graphs of order
$\ell p^2$. For the infinite family of graphs $\Sigma_{p,p,R}$
described in Lemma~\ref{lem-C2}, we have determined their structure and
established that they are edge-transitive normal Cayley graphs.
However, we have not determined the full automorphism groups of
these graphs. The symmetry properties of this family are closely
related to the specific structure of the group $R$. For instance,
when $R$ is abelian, the corresponding graphs are arc-transitive Cayley graphs
by Proposition~\ref{bi-Cayley1};
when $R$ is non-abelian, however, it is not clear in general
whether they are arc-transitive or not. We are pleased to leave
this as an open problem for interested readers.

\subsection{Proof of Theorem~\ref{the:square-free}}
The goal of this subsection is to prove Theorem~\ref{the:square-free}.
Firstly, we further characterize the connected cubic edge-transitive graph of order $\ell p^2$ that satisfies Theorem~\ref{the:lp2}~(2.4),
under the condition that $\ell$ is square-free.

\begin{lem}\label{lem-sqaurefree}
Under Assumption~I, let $p>3$. If $P\cong Z_p\times Z_p$ and $\ell$ is square-free, then one of the following holds:
\begin{enumerate}
  \item[{\rm (1)}] $\ell=2$ and $\G$ is isomorphic to an arc-transitive  graph as listed
  in {\rm\cite[Theorem 3.5(2)]{FK-AustMath}};
    \item[{\rm (2)}] $\ell\geq4$ and $\G\cong\Sigma_{p,p,R}$;
    \item[\rm (3)] $\ell=6$ and $\G$ is isomorphic to one of the graphs listed in~{\rm\cite[Theorem 5.3]{FK-2004}} or {\rm \cite[Theorem 4.2]{LWX}}.
\end{enumerate}
\end{lem}

\proof
By Lemma~\ref{lem-C2}, the (1) holds when $\ell=2$. Assume henceforth $\ell\geq 4$. Again by Lemma~\ref{lem-C2}, either $\G\cong\Sigma_{p,p,R}$ or $\G$ is a $C$-normal cover of one of the graphs listed in Lemma~\ref{lem-C2}~(3) with $C=C_A(P)=P\times R$. To complete the proof, we only need to consider the latter case. Since $|V(\G)|=\ell p^2=|V(\G_C)||C|=|V(\G_C)|\cdot |R|\cdot p^2$, we obtain $\ell=|V(\G_C)|\cdot|R|$, which in particular implies
that $|V(\G_C)|\mid \ell$ and so $|V(\G_C)|$ is square-free.
By Lemma~\ref{lem-C2}~(3), we have $\G_C=\F6$ and $\ell=6|R|$.
Moreover, $|R|$ is also square-free and $\gcd(6,|R|)=1$.
If $R=1$, then $\ell=6$ and $|V(\G)|=6p^2$, such graphs have been completely classified in~{\rm\cite[Theorem 5.3]{FK-2004}} and {\rm \cite[Theorem 4.2]{LWX}}, and so the (3) holds.

In the rest of the proof, we always suppose $R\neq 1$.
Let $\overline{A}=A/P$ and let $\overline{R}=C/P$.
Then $R\cong\overline{R}\unlhd\overline{A}$.
We first prove the following claim.

\medskip
\f{\bf Claim.}\ $R$ is cyclic, and $\G$ is arc- or edge-regular.
\medskip

Recall that $P\unlhd A$ is semiregular.
Then $\G_P$ is a cubic $\overline{A}$-edge-transitive graph of square-free order $\ell=6|R|$ (see Propositions~\ref{arc} and~\ref{edge}).
It follows that $\frac{3}{2}|V(\G_P)|$ divides $|\overline{A}|$. Let $X=\Aut(\G_P)$.
We consider two cases depending on whether $\G_P$ is arc-transitive.

If $\G_P$ is not arc-transitive, then it is semisymmetric. By {\rm \cite[Theorem 1]{Liu-Lu}}, $\G_P$ is isomorphic to one of the following graphs:
\begin{enumerate}
  \item [{\rm (i)}]\ $\S210$.
  \item [{\rm (ii)}]\ a bi-coset graph of order $6m$ with $X\cong Z_m:Z_3^2$, $m=p_1p_2\cdot\cdot\cdot p_t$ and the $p_i$ are distinct primes with $p_i\equiv 1 \pmod 3$.
  \item [{\rm (iii)}]\ a bi-coset graph of order $\frac{r(r^2-1)}{12}$ with $X=\PGL_2(r)$ and $r=\pm11\pmod {24}$.
  \item [{\rm (iv)}]\ a bi-coset graph of order $\frac{r(r^2-1)}{24}$ with $X=\PSL_2(r)$ and $r=\pm23\pmod {48}$.
  \item [{\rm (v)}]\ a bi-coset graph of order $\frac{r(r^2-1)}{4}$ with $X=Z_3\times\PSL_2(r)$ and $r=\pm11\pmod {24}$.
\end{enumerate}

If $\G_P\cong\S210$, then $X\cong S_7$ and $|R|=|V(\G_P)|/6=35$. This implies $S_7$ has a subgroup of order $35$, which is impossible by~Magma~\cite{BCP}.
Noting that $\G_P$ is $\overline{A}$-edge-transitive, we have $\frac{3\cdot|V(\G_P)|}{2}\mid |\overline{A}|$. For case (ii), $\G_P$ is edge-regular and so $A/P=X=Z_{m}:Z_3^2$, forcing that $\G$ is edge-regular. Since $|R|=|V(\G_P)|/6=m$, we have $R\cong Z_m$.
For case (iii), $\frac{3\cdot|V(\G_P)|}{2}=\frac{3\cdot r(r^2-1)}{2\cdot12}=\frac{3}{2\cdot6}|\PSL_2(r)|$ divides $|\overline{A}|$, where $\overline{A}\leq \PGL_2(r)$.
It follows from Lemma~\ref{lem-pgl} that $\overline{A}=\PSL_2(r)$ or $\PGL_2(r)$ as $r=\pm11\pmod {24}$.
In each case, $\overline{A}$ does not contain a normal subgroup $\overline{R}$ of order $|V(\G_P)|/6$, a contradiction.
For case (iv), $\frac{3\cdot r(r^2-1)}{2\cdot24}=\frac{3}{2\cdot12}|\PSL_2(r)|$ divides $|\overline{A}|$.
Since $r=\pm23\pmod {48}$, by Lemma~\ref{lem-pgl}, we have $\overline{A}=\PSL_2(r)$.
However, $\PSL_2(r)$ does not contain a normal subgroup of order $|V(\G_P)|/6$, a contradiction.
For case (v), either $\overline{A}\leq \PSL_2(r)$ or $\overline{A}\cong Z_3\times H$ with $H\leq \PSL_2(r)$, where $r=\pm11\pmod {24}$.
It follows that $\frac{3}{2\cdot2}|\PSL_2(r)|\mid |\overline{A}|$ or $\frac{r(r^2-1)}{8}\mid |H|$. By Lemma~\ref{lem-pgl}, we have $\overline{A}=\PSL_2(r)$ or $\overline{A}=Z_3\times\PSL_2(r)$.
In each case, $\overline{A}$ does not contain a normal subgroup of order $|V(\G_P)|/6$, a contradiction.

Now assume that $\G_P$ is arc-transitive. Recall that $|V(\G_P)|=\ell$ is square-free
and $\ell=6|R|>6$. According to {\rm \cite[Theorem 1]{Li-Lu-Wang}}, one of the following holds:
\begin{enumerate}
  \item [{\rm (i)}]\ $\G_P\cong \Cay(\langle a,b\rangle,\{ab,a^kb,b\})$ and $X\cong \langle a,b\rangle:Z_3$ where $\langle a,b\rangle\cong D_{2n}$, $o(a)=n>3$, $n\neq7$ and $o(b)=2$ where $k^2-k+1\equiv 0\pmod n$.
  \item [{\rm (ii)}]\ $\G_P$ is isomorphic to Tutte's $8$-cage, $|V(\G_P)|=30$ and $X\cong{\rm P\Gamma L}_2(9)$.
  \item [{\rm (iii)}]\ $X\cong\PSL_2(r)$ or $\PGL_2(r)$ for a prime $r\geq5$ such that
\begin{itemize}
\item $r\equiv \pm3(\mod~8)$ and $|V(\G_P)|=|\PSL_2(r)|/6$ or $|\PSL_2(r)|/2$;
\item $r\equiv \pm7(\mod~16)$ and $|V(\G_P)|=|\PSL_2(r)|/12$ or $|\PSL_2(r)|/4$;
\item $r\equiv \pm15(\mod~32)$ and $|V(\G_P)|=|\PSL_2(r)|/24$ or $|\PSL_2(r)|/8$;
\end{itemize}
  \item [{\rm (iv)}]\ $X\cong\PSL_2(r):D_{2m}$, and $|V(\G_P)|=|X|/4$ for a prime $r\geq5$ and $1<m=\frac{8t}{r(r^2-1)}$.
\end{enumerate}

If $\G_P\cong \Cay(\langle a,b\rangle,\{ab,a^kb,b\})$, then
$X\cong D_{2n}:Z_3$ and $|V(\G_P)|=\ell=6|R|=2n$, which yields
$n=3|R|$. Since $\ell=6|R|$ is square-free, we have
$\gcd(6,|R|)=1$, forcing $n=3|R|$ to be odd.
Note that $\overline{A}\leq X$ and $\frac{3}{2}|V(\G_P)|\mid |\overline{A}|$, where $\frac{3}{2}|V(\G_P)|=3n$.
Consequently, $\G_P$ is either $\overline{A}$-arc-regular with
$\overline{A}=X$, or $\overline{A}$-edge-regular with
$\overline{A}\cong\langle a\rangle:Z_3$. Hence, $\G$ is either
arc-regular or edge-regular.
Furthermore, since $R\cong\overline{R}\unlhd \overline{A}$ and
$\gcd(6,|R|)=1$, it follows that $R\cong\overline{R}\leq\langle a\rangle$
is cyclic.

For case~(ii), we have $|V(\G_P)|=30$ and $|R|=5$.
Since $\frac{3}{2}|V(\G_P)|\mid |\overline{A}|$ and
$R\cong\overline{R}\unlhd \overline{A}$, it follows that
$\overline{A}$ is a subgroup of ${\rm P\Gamma L}(2,9)$ such that
$45\mid |\overline{A}|$ and $\overline{A}$ possesses a normal
subgroup of order $5$. However, a Magma computation~\cite{BCP}
shows that ${\rm P\Gamma L}(2,9)$ has no such subgroup,
a contradiction.

For case (iii), we have $\overline{A}\leq\PGL_2(r)$ and $\frac{3}{2}|V(\G_P)|\mid |\overline{A}|$, where $\frac{3}{2}|V(\G_P)|=\frac{3}{2k}|\PSL_2(r)|$ with
$k\in\{2,4,6,8,12,24\}$.
By Lemma~\ref{lem-pgl}, we have $\overline{A}=\PSL_2(r)$ or $\PGL_2(r)$.
However, neither $\PSL_2(r)$ nor $\PGL_2(r)$ contains a normal subgroup of order $|R|=|V(\G_P)|/6$, a contradiction.

For case~(iv), we have $\frac{3m}{4}|\mathrm{PSL}(2,r)|\mid |\overline{A}|$,
and hence $|X:\overline{A}|\leq \frac{8}{3}$. This forces
$|X:\overline{A}|=1$ or $2$, and consequently
$\mathrm{PSL}(2,r)\unlhd \overline{A}$.
Since $\overline{R}\cap \mathrm{PSL}(2,r)\unlhd \mathrm{PSL}(2,r)$,
we have $\overline{R}\cap \mathrm{PSL}(2,r)=1$ or $\mathrm{PSL}(2,r)$.
If $\overline{R}\cap \mathrm{PSL}(2,r)=1$, then
$|X|\geq |\overline{R}\,\mathrm{PSL}(2,r)|=|\overline{R}|\,|\mathrm{PSL}(2,r)|$,
which implies that $|\overline{R}|\leq 2m$.
But $|\overline{R}|=|V(\G_P)|/6=2m|\mathrm{PSL}(2,r)|/24$, so
$|\mathrm{PSL}(2,r)|\leq 24$, impossible as $r\geq 5$.
Therefore, $\overline{R}\cap \mathrm{PSL}(2,r)=\mathrm{PSL}(2,r)$,
that is, $\mathrm{PSL}(2,r)\leq \overline{R}$.
This implies that $4\mid |\overline{R}|$, contradicting
$\gcd(6,|\overline{R}|)=\gcd(6,|R|)=1$.
This completes the proof of the claim.

\medskip
Recall that $|V(\G_C)|=\ell$ is square-free and $C=C_A(P)$.
By Lemma~\ref{lem-C2}~(3), there exists a subgroup $M$ of $A$ such that
$\G_C$ is $M/C$-edge-regular and $\G$ is $M$-edge-regular,
with $M/C\cong Z_3^2$. Thus $M=C.Z_3^2$.
Recall also that $C=P\times R$ with $P\cong Z_p^2$, and both $P$ and $R$
are normal in $A$. By the Claim, $R$ is cyclic, so we may write
$R\cong Z_k$. Let $Q$ be a Sylow $3$-subgroup of $M$. Since $p\geq 5$ and
$\gcd(6,k)=1$, we have $Q\cong Z_3^2$ and $M=CQ=C:Q$.
Set $G=PQ$. Then $G=P:Q$, and hence $C_G(P)=C\cap G=P$.
However, by Lemma~\ref{lem-centra}, we have $C_G(P)>P$, a contradiction.
This completes the proof of the lemma.\hfill\qed

We are now ready to prove Theorem~\ref{the:square-free}.

\medskip
{\noindent\bf Proof of Theorem~\ref{the:square-free}.}
Let $\ell:=2t$ be a square-free integer and let $p$ be a prime such that $p>\ell$.
By Theorem~\ref{the:cycle}, Lemma~\ref{lem-C2-normal-case-graph}, \cite[Theorem 5.3]{FK-2004} and~\cite[Theorem 4.2]{LWX},
we know that all graphs listed in Theorem~\ref{the:square-free} are connected cubic edge-transitive graphs of order $\ell p^2$.

Conversely, let $\G$ be a connected cubic edge-transitive graph of order $\ell p^2$ satisfying Assumption~I. By Theorem~\ref{the:lp2}, one of the following holds:
\begin{itemize}
  \item[{\rm (i)}] $p=3$ and $\G\cong \F_{18}$.
  \item[{\rm (ii)}] $p\geq5$, $P\cong Z_{p^2}$, $\G$ is arc-regular and $\G\cong\G_{p,R_1}$, or $\G$ is edge-regular and $\G\cong\G_{p,R_2}$; or
  \item[{\rm (iii)}] $p\geq5$, $P\cong Z_p^2$.
\end{itemize}

For case (ii), we have $\ell=2|R_i|$ with $1\leq i\leq2$, and $|R_i|$ is an odd square-free integer. Thus, the (2.1) holds. For case (iii), Lemma~\ref{lem-sqaurefree} yields the (2.2), (2.3) and (2.4), according as the alternatives in Lemma~\ref{lem-sqaurefree} are (1), (2) and (3).
This completes the proof.
\hfill\qed

\medskip

\f {\bf Acknowledgements:}\ The first and second authors were partially supported by the National Natural Science Foundation of China (Nos. 12401463 and 12501469, respectively). The third author were partially supported by the National Natural Science Foundation of China (Nos. 12101070 and 12571370).

\section{Appendix: Proof of Lemma~\ref{lem-simple-p2}}

We will give the proof of Lemma~\ref{lem-simple-p2} in this section.
Firstly, we need the following results about non-abelian simple groups with an abelian Sylow subgroup of odd order.

\begin{prop}[{\rm\cite[Theorem]{Shen-Zhou}}]\label{prop=1}
Let $S$ be a finite non-abelian simple group.
Suppose that $S$ has an abelian Sylow $p$-subgroup with $p$ an odd prime.
Then $S$ is one of the following groups:
\begin{enumerate}
\item[\rm (1)] an alternating group $\A_n$ with $n<p^2$;
\item[\rm (2)] a linear simple group $\PSL_2(q)$ with $p\bigm| q$;
\item[\rm (3)] $\PSL_3(q)$, $q\equiv 4,7\pmod 9$ with $p=3$;
\item[\rm (4)] $\PSU_3(q^2)$, $2<q\equiv 2,5\pmod 9$ with $p=3$;
\item[\rm (5)] a sporadic simple group listed in {\rm\cite[Table~4]{Shen-Zhou}};
\item[\rm (6)] a simple classical group over the Galois field $\GF(q)$ except the above groups of $(3)$ and $(4)$, where $p$ is a primitive prime of $q^m-1$ and $e_L(mp)=0$, and the function $e_L(x)$ is defined in Table~{\rm\ref{table=2}};
\item[\rm (7)] a simple exceptional group over the Galois field $\GF(q)$, where $p$ is a primitive prime of $q^m-1$ and $e_L(mp)=0$, and the possible values for $m$ and $e_L(m)$ are given in Table~{\rm\ref{table=3}}.

\end{enumerate}
\end{prop}

\begin{table}[!htp]
\caption{Simple classical groups $L$ and its function $e_L(m)$}\label{table=2}
\begin{center}
\newcommand{\tabincell}[2]{\begin{tabular}{@{}#1@{}}#2\end{tabular}}
\begin{tabular}{|l|l|}
\hline
 $L$ & $e_L(m)$, where $m$ is a positive integer\\
\hline
$\PSL_n(q)$ & $e_L(1)=n-1$ or $e_L(m)=[\frac{n}{m}]$ with $m>1$\\
\hline
$\PSU_n(q^2)$ & \tabincell{l}{$e_L(2)=n-1$,\\ $e_L(m)=[\frac{n}{\lcm(2,m)}]$ with $m\not\equiv 2\ (\mod 4)$,\\ $e_L(m)=[\frac{2n}{m}]$ with $2<m\equiv 2\ (\mod 4)$}\\
\hline
$\PSp_{2n}(q),\P\O_{2n+1}(q)$& $e_L(m)=[\frac{2n}{\lcm(2,m)}]$\\
\hline
$\P\O^+_{2n}(q)$& \tabincell{l}{$e_L(m)=[\frac{2n}{\lcm(2,m)}]$ with $m\bigm| n$ or $m\nmid 2n$\\ $e_L(m)=\frac{2n}{m}-1$ with $m\nmid n$ or $m\bigm| 2n$}\\
\hline
$\P\O^-_{2n}(q)$& \tabincell{l}{$e_L(m)=[\frac{2n}{\lcm(2,m)}]-1$ with $m\bigm| n$,\\ $e_L(m)=[\frac{2n}{\lcm(2,m)}]$ with $m\nmid n$}\\
\hline
\end{tabular}
\end{center}
\end{table}

\begin{table}[!ht]
\caption{Simple exceptional groups $L$ and its function $e_L(m)$}\label{table=3}
\begin{center}
\newcommand{\tabincell}[2]{\begin{tabular}{@{}#1@{}}#2\end{tabular}}
\begin{tabular}{|l|l|l|l|l|l|l|l|l|l|l|l|}
\hline
\diagbox[width=2em,trim=l]{$m$}{$L$} & $^2B_2$ & $^3D_4$ & $G_2$ & $^2G_2$& $F_4$ & $^2F_4$& $E_6$ & $^2E_6$& $E_7$ &$E_8$\\
\hline
1 & 1 &2 & 2 &1 & 4 &2 &6  & 4& 7 &8\\
\hline
2 &  &2 & 2 &1 & 4 &2 &4  & 6& 7 &8\\
\hline
3 &  &2 & 2 & & 2 & &3  & 2& 3 &4\\
\hline
4 & 1 & &  & & 2 &2 &2  & 2& 2 &4\\
\hline
5 &  & &  & &  & &1  & & 1 &2\\
\hline
6 &  &2 & 1 &1 & 2 &1 &2  & 3& 3 &4\\
\hline
7 &  & &  & &  & &  & & 1 &1\\
\hline
8 &  & &  & & 1 &1 &1  & 1& 1 &2\\
\hline
9 &  & &  & &  & &1  & & 1 &1\\
\hline
10 &  & &  & &  & &  & 1& 1 &2\\
\hline
12 &  &1 &  & & 1 &1 &1  & 1& 1 &2\\
\hline
14 &  & &  & &  & &  & & 1 &1\\
\hline
15 &  & &  & &  & &  & &  &1\\
\hline
18 &  & &  & &  & &  & 1& 1 &1\\
\hline
20 &  & &  & &  & &  & &  &1\\
\hline
24 &  & &  & &  & &  & &  &1\\
\hline
30 &  & &  & &  & &  & &  &1\\
\hline
\multicolumn{11}{|l|}{\footnotesize\tabincell{l}{{\bf Remark.}\  The entry in the $m$-row
and the $L$-column is $e_L(m)$ if non-zero\\ and blank otherwise.} }\\
\hline
\end{tabular}
\end{center}
\end{table}

According to Proposition~\ref{prop=1}, we will divide the proof of Lemma~\ref{lem-simple-p2} into two lemmas:
Lemma~\ref{lem-non-Lie-type}, which addresses cases (1)--(5) outlined in Proposition~\ref{prop=1},
and Lemma~\ref{lem-Lie-type}, which tackles the remaining situations in Proposition~\ref{prop=1}.

\begin{lem}\label{lem-non-Lie-type}
Let $p$ be an odd prime and let $T$ be a finite non-abelian simple group with a Sylow $p$-subgroup, say $P$, of order $p^2$.
If $|T|<192\cdot p^3$ and $T$ is one of the groups appearing in {\rm Proposition~\ref{prop=1}~(1)--(5)},
then $(T, P)=(A_6,Z_3^2), (A_7,Z_3^2), (\PSL_2(9),Z_3^2), (\PSL_2(25),Z_5^2)$ or $(\PSL_2(49), Z_7^2)$.
\end{lem}

\f\demo By the assumptions, $T$ is isomorphic to one of the groups given in Proposition~\ref{prop=1}~(1)--(5).

First let $T=A_n$ with $n<p^2$. Clearly, $A_n$ has no $p^2$-cycle, so $P\cong Z_p^2$.
Then we will have $2p\leq n<3p$ (see also \cite[Corollary~2.2]{Shen-Zhou}). Since $p\geq3$, one has
\[\frac{|A_n|}{p^3}=\frac{n!}{2p^3}\geq \frac{(2p)!}{2p^3}=\frac{(2p-3)!(2p-2)(2p-1)(2p)}{2p^3}>(2p-3)!.\]
If $p\geq5$, then $(2p-3)!\geq 7!=5040>192$, and so $|A_n|>192\cdot p^3$, a contradiction.
Thus, $p=3$. Since $6=2p\leq n<3p=9$, one has $n=6, 7$ or $8$, and since $|T|<192\cdot 3^3$, one has $n=6$ or $7$.
It is easy to check that $\frac{|A_6|}{3^3}=\frac{40}{3}<192$ and $\frac{|A_7|}{3^3}=\frac{280}{3}<192$, one has $T=A_6$ or $A_7$.

Next let $T=\PSL_2(q)$ with $p\bigm| q$. Note that $|\PSL_2(q)|=\frac{q(q^2-1)}{2}$.
Since $p\bigm| q$, every Sylow $p$-subgroup of $\PSL_2(q)$ has order $q$, and so $q=p^2$ since $P$ is a Sylow $p$-subgroup of $\PSL_2(q)$ of order $p^2$.
Then
\[192>\frac{|\PSL_2(p^2)|}{p^3}=\frac{p^2(p^4-1)}{2p^3}=\frac{p^4-1}{2p}>p^2+1.\]
From $p^2+1<192$ it follows that $p\leq 13$, and then $p=3,5$ or $7$ due to $\frac{p^4-1}{2p}<192$.
Consequently, we have $T=\PSL_2(9)$, $\PSL_2(25)$ or $\PSL_2(49)$,
and by Magma~\cite{BCP}, $P\cong Z_3^2$, $Z_{5}^2$ or $Z_7^2$, respectively.

Now let $T=\PSL_3(q)$, $q\equiv 4,7\pmod 9$ with $p=3$.
Then \[|T|=|\PSL_3(q)|=\frac{q^3(q^2-1)(q^3-1)}{\gcd(3, q-1)}\geq q^2(q^2-1)(q^3-1).\]
Since $q\equiv 4,7\pmod 9$, one has $q\geq 4$ and so $|T|\geq q^2(q^2-1)(q^3-1)\geq 15120$.
It then follows that $\frac{|T|}{3^3}=\frac{|\PSL_3(q)|}{3^3}\geq\frac{15120}{3^3}>192$. A contradiction occurs.

Let $T=\PSU_3(q^2)$, $2<q\equiv 2,5\pmod 9$ with $p=3$. Then
\[|T|=|\PSU_3(q^2)|=\frac{q^6(q^4-1)(q^6+1)}{\gcd(3,q+1)}\geq \frac{q^6(q^4-1)(q^6+1)}{3}>\frac{2^6(2^4-1)(2^6+1)}{3}=20800,\]
and then $\frac{|T|}{3^3}>\frac{20800}{27}>192$, which is contrary to the assumption that $|T|<{192\cdot p^3}$.

Now let $T$ be a sporadic simple group. By Atlas~\cite{A}, it is easy to see that if $p\geq 17$ and $p\bigm| |T|$,
then every Sylow $p$-subgroup of $T$ has order $p$, a contradiction.
If $p\leq 13$, again by checking Atlas~\cite{A} (see also \cite[Table~4]{Shen-Zhou}), $(T, p)$ is one of the following pairs:
\begin{align*}
&(M_{k},3) (k=11,22,23), (J_{2},5), (HS,3), (Suz,5),(He,5), \\
&(Fi_{k}, 5)(k=22,23), (Th,7), (Co_1,7), (Fi^{'}_{24},5), (B,7), (M,11).
\end{align*}
However, for any one of these 13 pairs, we have $|T|>192\cdot p^3$, a contradiction.\hfill\qed

For two integers $m$ and $n$, we use $m\parallel n$ to denote $m\bigm| n$ but $m^2\nmid n$.

\begin{lem}\label{lem-Lie-type}
Let $p$ be an odd prime and let $T$ be a finite non-abelian simple group with a Sylow $p$-subgroup, say $P$, of order $p^2$.
If $|T|<192\cdot p^3$ and $T$ is one of the groups appearing in Proposition~{\rm\ref{prop=1}~(6)--(7)}, then $(T,P)=(\PSL_2(8), Z_9), (\PSL_2(17), Z_9)$ or $(\PSL_2(19), Z_9)$.
\end{lem}

\f\demo Suppose that $|T|<192\cdot p^3$ and that $T$ is one of the groups appearing in Proposition~\ref{prop=1}~(6)-(7).
In this case, $p$ is a primitive prime divisor of $q^m-1$, that is, $p\bigm| (q^m-1)$ and $p\nmid(q^{i}-1)$ for all $i<m$.
By our assumption, $P$ is a Sylow $p$-subgroup of $T$ of order $p^2$, and by \cite[Corollary~3.9]{Shen-Zhou}, we have $P\cong Z_{p^t}^{e_L(m)}$. So we have the following claim.\medskip

\f{\bf Claim 1.}\ Either $t=2$, $e_L(m)=1$ and $P\cong Z_{p^2}$, or $t=1$, $e_L(m)=2$ and $P\cong Z_p^2$.\medskip

Clearly, for any positive integer $k$, we have $\gcd(\frac{q^k-1}{2},\frac{q^k+1}{2})=1$;
this gives the next claim, which will be repeatedly used in the subsequent proof.

\f{\bf Claim 2.}\ If $m$ is even, then $p\bigm| (q^{m/2}+1)$, and if further $p\parallel (q^m-1)$, then $p\parallel (q^{m/2}+1)$.\medskip

We shall divide the proof into the following two cases.

\medskip
\f{\bf Case~1.}\ $T=\PSL_n(q)$.\medskip

In this case, we shall prove that $(T,P)=(\PSL_2(8), Z_9), (\PSL_2(17), Z_9)$ or $(\PSL_2(19), Z_9)$.
By checking Table~\ref{table=2}, either $e_L(m)=n-1$ for $m=1$, or $e_L(m)=[\frac{n}{m}]$ for $m>1$.

Assume first that $m=1$. Then $e_L(1)=n-1=1$ or $2$ by Claim~1, and then $n=2$ or $3$.
Let $n=2$.
Recall that $p$ is a primitive prime divisor of $q-1$ and $p\geq 3$.
Then $p\nmid q(q+1)$, and so $p^2\parallel (q-1)$, forcing  $p<\sqrt{q}$.
In particular, $q\geq11$.
By our assumption, we have
\[192>\frac{|\PSL_2(q)|}{p^3}>\frac{\sqrt{q}(q+1)}{\gcd(2, q-1)}
\geq\frac{\sqrt{q}(q+1)}{2},\]
implying that $q\leq 49$.
Since $p\geq3$ and $p^2\parallel (q-1)$, we obtain $(q,p)=(19,3) $ or $(37,3)$.
Note that $|\PSL_2(37)|/3^3=937>192$.
Thus, $(T,P)=(\PSL_2(19), Z_9)$.
Let $n=3$. Then $p\geq3$ and $p\bigm| (q-1)$ implies that $q\geq 4$ and $p<q$. Again by our assumption, we have
\[192>\frac{|\PSL_3(q)|}{p^3}=\frac{{q^3}(q^3-1)(q^2-1)}{\gcd(3, q-1)p^3}
>\frac{{q^2(q^2+q+1)(q+1)}}{3}\geq \frac{16\cdot 21\cdot 5}{3}=560,\]
which is impossible.

Assume now that $m>1$. By Claim~1, we have $e_L(m)=[\frac{n}{m}]=1$ or $2$.
If $e_L(m)=2$, then $4\leq 2m\leq n<3m$ and $p\bigm| (q^m-1)$.
So $p^3\leq (q^{2m}-1)(q^m-1)$. It follows that
\[\frac{|\PSL_n(q)|}{p^3}=\frac{q^{(n-1)n/2}}{\gcd(n,q-1)}\cdot\frac{(q^{2m}-1)(q^m-1)}{p^3}\cdot(q^3-1)\prod_{i\neq 2m,m,3,1}(q^i-1)\geq q^{5}(q^3-1)\geq2^5\cdot 7>192,\]
a contradiction.

We now consider the case when $e_L(m)=1$. In this case, we have $m\leq n<2m$ and $p\bigm| q^m-1$ but $p\nmid q^\ell-1$ for any $\ell<m$.
In particular, $p\nmid q$.
If $p\bigm|(q^i-1)$ for some $i>m$, then $p$ divides $q^{i-m}$, which is impossible.
Thus, $p\nmid(q^i-1)$ for all $m<i\leq n$.
Since $p^2\parallel|T|$, one has $p^2\parallel(q^m-1)$.
Since $m>1$, one has $m-1\geq m/2$, and then $q^{m-1}\geq q^{m/2}>\sqrt{p^2+1}>p$. It follows that $q^{m-1}-1\geq p$.

Let $m=2$.
Then $p^2\parallel (q^2-1)$, and so $p^2\parallel (q+1)$ as $p\nmid (q-1)$.
It follows that $3\leq p\leq \sqrt{q+1}$, and so $q\geq 8$.
If $n>2$, then \[\frac{|\PSL_n(q)|}{p^3}>\frac{q^{n(n-1)/2}}{\gcd(n, q-1)p}\cdot \frac{q^2-1}{p^2}\cdot(q^n-1)> q^2(q^3-1)>192,\]
a contradiction.
If $n=2$, then we have
\[192>\frac{|\PSL_2(q)|}{p^3}\geq\frac{q(q-1)}{2\sqrt{q+1}}>\frac{(q-1)\sqrt{q-1}}{2}.\]
It follows that $q<54$.
As $p^2\parallel (q+1)$ and $p\geq3$, we have $(q,p)=(8,3),(17,3)$ or $(47,5)$.
Notice that $\frac{|\PSL_2(47)|}{5^3}=470$.
Then $(T,P)=(\PSL_2(8),Z_9)$ or $\PSL_2(17),Z_9)$, as required.

Let $m=3$.
Then $n\geq3$ and $p^2\parallel (q^3-1)$, which means that $p^2\parallel(q^2+q+1)$ as $p\nmid(q-1)$.
It follows that $p<q+1$ and $p^3\leq (q^3-1)(q+1)$.
In particular, $q\geq3$.
Thus, we have
$$192>\frac{|\PSL_n(q)|}{p^3}> q^{n(n-1)/2}\frac{(q-1)^2}{\gcd(n,q-1)}\cdot\frac{(q^2+q+1)(q+1)}{p^3}>q^3(q-1),$$
implying that $q=3$ and $p=13$, which is impossible as $p<q+1$.

Let $m=4$. Then $n\geq4$ and $p^2\parallel (q^4-1)$.
Thus, $p^2\parallel (q^2+1)$ and so $p<q+1$.
This gives rise to
\[\frac{|\PSL_n(q)|}{p^3}\geq q^{n(n-1)/2}\cdot\frac{(q^4-1)}{p^2}\cdot\frac{(q^2-1)}{\gcd(n,q-1)p}\cdot(q^3-1)>q^6(q^3-1)(q^2-1)>192,\]
a contradiction.

Let $m\geq 5$. Then $n\geq 5$, and $p^2\parallel q^m-1$, implying that
\[\frac{|\PSL_n(q)|}{p^3}> \frac{q^{\frac{n(n-1)}{2}-1}}{p}(q^2-1)\geq \frac{q^9(q^2-1)}{3}>192,\]
a contradiction.

\medskip
\f{\bf Case~2.}\ $T\ncong\PSL_n(q)$.\medskip

In this case, we shall prove that every Lie type simple group in Proposition~\ref{prop=1}~(6)--(7) which is not isomorphic to $\PSL_n(q)$ does not happen.

\medskip
\f{\bf Subcase~2.1.}\ $T=\PSU_n(q) (n\geq 3)$.\medskip

By checking Table~\ref{table=2}, we have
\begin{equation}\label{eq-psu}
e_L(m)=\left\{
                \begin{array}{ll}

[\frac{n}{\lcm(2,m)}], & \hbox{if}\  m{~\not\equiv~} 2\pmod{4}; \\

                  n-1, & \hbox{if}\ m=2; \\

                  [\frac{2n}{m}], & \hbox{if}\ 2<m\equiv 2\pmod{4}.

                \end{array}
              \right.\end{equation}
Furthermore,
$|\PSU_n(q)|=\frac{q^{(n-1)n/2}}{\gcd(n, q+1)}\prod_{i=2}^{n}(q^{i}-(-1)^i)$.
By Claim~1, under our assumption we have either $e_L(m)=2$ or $1$.

First assume that $e_L(m)=2$. By Eq.~(\ref{eq-psu}), if $m=2$ then $n=3$. Recall that $p^2\parallel |\PSU_3(q)|$ and $p$ is an odd primitive prime of $q^2-1$. It follows that $p\parallel (q+1)$, and so $q\geq 4$. If  $q=4$ then $p=5$, but $\frac{|\PSU_3(4)|}{5^3}=499.2>192$, a contradiction.
If $q=5$, then $p=3$, but $\frac{|\PSU_3(5)|}{3^3}>192$, a contradiction.
If $q>5$, then $p\parallel (q+1)$ implies that $q\geq 11$. Clearly, $q^3\geq 3(q+1)\geq 3p$. Then
\[\frac{|\PSU_3(q)|}{p^3}\geq (q^2-q+1)(q-1)>192,\]
a contradiction.
Therefore, $m\neq2$.

Assume that $m\not\equiv 2\pmod 4$, then by Eq.~\eqref{eq-psu}, $e_L(m)=2$ implies that $2\cdot\lcm(2,m)\leq n<3\cdot \lcm(2,m)$.
If $m=1$, then Eq.~\eqref{eq-psu} means that $n=4$, and so $p\bigm| (p-1)$.
Thus, $p^3\leq (p-1)^3<(p+1)(p-1)^2$, yielding  that
$$\frac{|\PSU_4(q)|}{p^3}=q^6(q^3+1)\frac{(q^4-1)(q^2-1)}{\gcd(4,q-1)p^3}>q^6(q^3+1)(q^2+1)>192,$$
a contradiction.
Thus, $m\geq3$, and so $n\geq 2m>m+2$ and $n\geq 6$.
Clearly, $p^3\leq (p^{m+2}+1)(p^{m+1}-1)(p^m+1)$.
Then
\[\frac{|\PSU_n(q)|}{p^3}\geq \frac{q^{n(n-1)/2}}{q+1}\cdot (q^{2m}-1)(q^2-1)\geq q^{13}(q^6-1)(q^2-1)>192,\]
a contradiction.

Clearly, $\gcd(\frac{q^m-1}{2},\frac{q^m+1}{2})=1$.
Then $p\bigm| (q^m-1)$ and $m$ is even means that $p\bigm| (q^{m/2}+1)$.

If $2<m\equiv 2\pmod 4$, then by Eq.~(\ref{eq-psu}), $e_L(m)=2$ implies that $2m\leq 2n<3m$.
In this case, we have $n\geq m\geq 6$, and $m-2>m/2$.
By Claim~2, $p\bigm| (q^{m/2}+1)$.
So $p^3\leq (q^{m/2}+1)(q^{m-1}+1)(q^{m-2}-1)$. It follows that
\[\frac{|\PSU_n(q)|}{p^3}\geq  q^{n(n-1)/2-2}(q^{m/2}-1)(q^3+1)(q^2-1)\geq q^{13}(q^6-1)(q^2-1)>192,\]
a contradiction.

Now assume that $e_L(m)=1$. Since it is assumed that $n\geq 3$, by Eq.~(\ref{eq-psu}), we have $m\neq 2$.

If $m\not\equiv 2\pmod 4$, then by Eq.~(\ref{eq-psu}), $e_L(m)=1$ implies that $\lcm(2,m)\leq n<2\cdot \lcm(2,m)$.
Suppose first that $m$ is odd. Then $\lcm(2,m)=2m$.
If $m=1$, then $n=3$ and $p\bigm|(q-1)$, and so
$$\frac{|\PSU_3(q)|}{p^3}>\frac{q^3(q^2-q+1)}{3}>192,$$
a contradiction.
Thus $m\geq3$, namely, $n\geq 2m\geq6$ and $2m>m+2$.
Clearly, $p^3\leq (p^{m+2}+1)(p^{m+1}-1)(p^m+1)$.
Then
\[\frac{|\PSU_n(q)|}{p^3}\geq \frac{q^{n(n-1)/2}}{q+1}\cdot (q^{2m}-1)(q^2-1)\geq q^{13}(q^6-1)(q^2-1)>192,\]
a contradiction.
Suppose now that $m$ is even. Then $\lcm(2,m)=m$, $4\bigm| m$ and $m\leq n<2m$.

Let $n=m$.
Then $p^2\parallel(q^m-1)$, and so $p^2\parallel(q^{m/2}+1)$.
This indicates that
$$ \frac{|\PSU_n(q)|}{p^3}\geq \frac{q^{n(n-1)/2}}{\gcd(n,q+1)}\cdot\frac{(q^m-1)(q^3+1)(q^2-1)}{p^3}>\frac{q^6(q^3+1)(q^2-1)(q-1)}{3}>192,$$
a contradiction.
Thus, $n\geq m+1\geq5$. In the case, we get $q^{m+1}+1>q(q^m-1)$, and hence
$$\frac{|\PSU_n(q)|}{p^3}\geq\frac{q^{n(n-1)/2}}{\gcd(n,q+1)}\cdot\frac{(q^{m+1}+1)(q^{m}-1)(q^2-1)}{p^3}>\frac{q^{11}(q^2+1)^2}{3}>192,$$
also a contradiction.

If $2<m\equiv 2\pmod 4$, then by Eq.~(\ref{eq-psu}), $e_L(m)=1$ implies that $m\leq 2n<2m$.
Since $m\equiv 2 \pmod{4}$, $m/2$ is odd, and $p\bigm| (q^{m/2}+1)$ by Claim~2, and furthermore, $p\nmid (q^{i}-1)$ for any $2\leq i\leq n$ since $n<m$.
Suppose first that $p\parallel (q^{m/2}+1)$.
Then $p^2\parallel |\PSU_n(q)|$ implies that $p\parallel (q^{i}+1)$ for some $i\leq n$ and $i\neq \frac{m}{2}$.
Since $p$ is a primitive prime of $q^m-1$, one has $i\geq m/2$,
and so $p\bigm| (q^i-q^{m/2})$, implying that $p\bigm| (q^{i-m/2}-1)$.
This leads to $i-m/2\geq m$ and so $n\geq i\geq 3m/2>m$ which is impossible.
Therefore, $p^2\parallel (q^{m/2}+1)$.
If $m/2\geq4$, then $n\geq m/2\geq 4$ and then $n(n-1)/2\geq n+2$. It follows that \[\frac{q^{n(n-1)/2}}{\gcd(n, q+1)p}\geq \frac{q^{n(n-1)/2}}{q^2\sqrt{q^n+1}}\geq \frac{q^{n(n-1)/2}}{q^{2+n}}\geq 1.\]
Since $m/2$ is odd, one has $n\geq m/2\geq 5$, and so
\[\frac{|\PSU_n(q)|}{p^3}\geq (q^4-1)(q^3+1)(q^2-1)\geq 192,\]
which is contrary to our assumption.
If $m/2=3$, then $p^2\parallel (q^{3}+1)$.
If $n\geq4$, then we have
$$\frac{|\PSU_n(q)|}{p^3}\geq \frac{q^6(q^4-1)(q^3+1)(q^2-1)}{\gcd(n,q+1)p^3}\geq\frac{q^6(q^4-1)(q^2-1)}{12}>192.$$
a contradiction.
Thus, $n=3$, and so
$$192>\frac{|\PSU_3(q)|}{p^3}= \frac{q^3(q^3+1)(q^2-1)}{\gcd(3,q+1)p^3}\geq\frac{q^3(q^2-1)}{9},$$
which implies that $q\leq4$, contradicting to $p^2\parallel(q^{3}+1)$.

\medskip
\f{\bf Subcase~2.2.}\ $T=\PSp_{2n}(q)$ or $\P\O_{2n+1}(q) (n\geq 2)$.\medskip

In this case, by checking Table~\ref{table=2}, we have $e_L(m)=[\frac{2n}{\lcm(2,m)}]$. Also, note that
\[|T|=\frac{q^{n^2}}{\gcd(2,q-1)}\prod_{i=1}^{n}(q^{2i}-1).\] Keep in mind that under our assumption we have either $e_L(m)=1$ or $2$.

We first consider the case when $e_L(m)=[\frac{2n}{\lcm(2,m)}]=1$. Let $m$ be even. Then $m\leq 2n<2m$.
As $p$ is an odd primitive prime of $q^m-1$,
if $p\bigm| (q^{2k}-1)$ for some $k\leq n$, then $2k\geq m$ and $p\bigm| (q^{2k-m}-1)$, and so either $2k-m=0$ or $2k-m\geq m$. However, since $2n<2m$, one has $2k=m$.
It follows that $p^2 \parallel (q^m-1)$ and by Claim~2, $p^2\parallel (q^{m/2}+1)$. Then $p^3\leq q^{m/2}(q^{m/2}+1)$.
Since $m>n\geq 2$,
we have $m\geq 4$ and so
\[\frac{|T|}{p^3}\geq \frac{q^{n^2}(q^m-1)(q^2-1)}{2q^{m/2}(q^{m/2}+1)}=\frac{1}{2}q^{n^2-m/2}(q^{m/2}-1)(q^2-1)\geq \frac{1}{2}q^{(m^2-2m)/4}(q^{m/2}-1)(q^2-1).\]
If $m=4$, then $p^2\parallel (q^2+1)$ implies that $q\geq 7$, and then the above inequality implies that $\frac{|T|}{p^3}>192$, a contradiction.
If $m\geq6$, then the above  inequality implies that $\frac{|T|}{p^3}\geq \frac{1}{2}q^{6}(q^{3}-1)(q^2-1)>192$, a contradiction.
Now let $m$ be odd. Then $e_L(m)=[\frac{2n}{\lcm(2,m)}]=1$ implies that $2m\leq 2n<4m$.
Since $n\geq 2$, we have $n\geq m\geq 3$.
Moreover, $p\bigm|(q^m-1)$ means that $p^2\leq q^{2m}-1$, and so
$$\frac{|T|}{p^3}\geq \frac{1}{2p^3}q^{n^2}(q^{2m}-1)(q^4-1)(q^2-1)>\frac{1}{6}q^{9}(q^4-1)(q^2-1)>192,$$
a contradiction.

We now consider the case when $e_L(m)=[\frac{2n}{\lcm(2,m)}]=2$. Assume that $m$ is even. Then $2m\leq 2n<3m$.
Clearly, $(q^{2m}-1)(q^m-1)\bigm| |T|$.
Since $p^2\parallel |T|$, one has $p\parallel (q^m-1)$.
If $m=2$, then $n=2$ as $2m\leq 2n<3m$.
Since $T$ is simple, one has $q>2$, and then
$$\frac{|T|}{p^3}\geq\frac{q^4(q^4-1)(q^2-1)}{p^3}\geq\frac{1}{6}q^4(q^2+1)(q-1)^2>3^3\cdot 20>192,$$
a contradiction.
Thus, $m\geq 4$, and by Claim~2, $p\parallel (q^{m/2}+1)$.
Then we always have $p^3\leq q^{m/2}(q^{m/2}+1)^2$. It follows that
\[\frac{|T|}{p^3}\geq \frac{1}{2}q^{n^2-m/2}(q^m+1)(q^{m/2}-1)^2\geq 2^{13}\cdot 15\cdot 9>192,\]
a contradiction.  Assume now that $m$ is odd. Then $4m\leq 2n<6m$. Clearly, $(q^{4m}-1)(q^{2m}-1)\bigm| |T|$.
Since $p^2\parallel |T|$, one has $p\parallel (q^{m}-1)$, and so $p^3\leq q^m(q^m-1)^2$. It follows that
\[\frac{|T|}{p^3}\geq \frac{1}{2}q^{n^2-m}(q^{2m}+1)(q^{m}+1)^2.\]
If $m\geq 3$, then the above inequality implies that $\frac{|T|}{p^3}>192$, a contradiction.
If $m=1$, then $n=2$. Since $p\parallel (q-1)$ and $p$ is odd, we have $q\geq4$, and again, the above inequality implies that $\frac{|T|}{p^3}>192$, a contradiction.

\medskip
\f{\bf Subcase~2.3.}\ $T=\P\O^{+}_{2n}(q)$ ($n\geq 4$).\medskip

In this case, by checking Table~\ref{table=2}, we have
\begin{equation}\label{eq-po+}
e_L(m)=\left\{
                \begin{array}{ll}

[\frac{2n}{\lcm(2,m)}], & \hbox{if}\  m\bigm| n\ {\rm or}\ m\nmid 2n;\\

\frac{2n}{m}-1, & \hbox{if}\  m\nmid n\ {\rm or}\ m\bigm| 2n.

                \end{array}
              \right.\end{equation}
Also, note that
\[|T|=\frac{1}{\gcd(4,q^n-1)}q^{n(n-1)}(q^n-1)
\prod_{i=1}^{n-1}(q^{2i}-1).\] Keep in mind that under our assumption we have either $e_L(m)=1$ or $2$.

If $m$ is even, then by Eq.~(\ref{eq-po+}), we have $m/2\leq n$,
and by Claim~2, one has $p\bigm| (q^{m/2}+1)$.
Then $p^3\leq q^{2(m/2+1)}(q^{m/2}+1)\leq q^{m+2}(q^{2n-2}-1)$,
and so $n\geq4$ means that
\[\frac{|T|}{p^3}\geq \frac{q^{n(n-1)}(q^n-1)\prod_{i=1}^{n-1}(q^{2i}-1)}{4q^{m+2}(q^{2n-2}-1)}\geq \frac{1}{4}q^{n(n-1)-m-2}(q^n-1)(q^{4}-1)(q^2-1)\geq 15\cdot 15\cdot 3>192,\]
a contradiction.

If $m$ is odd, then by Eq.~(\ref{eq-po-}), we have $m\leq n$. Again, since $p$ is an odd primitive prime of $q^m-1$, one has $p\bigm| (q^{m}-1)$.
Then $p^3\leq q^m(q^m-1)^2\leq q^m(q^n-1)(q^{2n-2}-1)$, and then
\[\frac{|T|}{p^3}\geq \frac{q^{n(n-1)}(q^n-1)\prod_{i=1}^{n-1}(q^{2i}-1)}{4q^m(q^n-1)(q^{2n-2}-1)}\geq \frac{1}{4}q^{n(n-1)-m}(q^{4}-1)(q^2-1)\geq 2^{6}\cdot 15\cdot 3>192,\]
a contradiction.

\medskip
\f{\bf Subcase~2.4.}\ $T=\P\O^{-}_{2n}(q)$ ($n\geq 4$).\medskip

In this case, by checking Table~\ref{table=2}, we have
\begin{equation}\label{eq-po-}
e_L(m)=\left\{
                \begin{array}{ll}

[\frac{2n}{\lcm(2,m)}]-1, & \hbox{if}\  m\bigm| n;\\

[\frac{2n}{\lcm(2,m)}], & \hbox{if}\  m\nmid n.

                \end{array}
              \right.\end{equation}
Also, note that
\[|\P\O^{-}_{2n}(q)|=\frac{1}{\gcd(4,q^n+1)}q^{n(n-1)}(q^n+1) \prod_{i=1}^{n-1}(q^{2i}-1).\] Keep in mind that under our assumption we have either $e_L(m)=1$ or $2$.
If $m$ is even, then by Eq.~(\ref{eq-po-}), we have $m/2\leq n$, and since $p$ is an odd primitive prime of $q^m-1$, one has $p\bigm| (q ^{m/2}+1)$.
Then $p^3\leq q^{m/2+1}(q^{m/2}+1)^2\leq q^{m/2+1}(q^n+1)(q^{2n-2}-1)$, and then
\[\frac{|T|}{p^3}\geq \frac{q^{n(n-1)}(q^n+1)\prod_{i=1}^{n-1}(q^{2i}-1)}{4q^m(q^n+1)(q^{2n-2}-1)}\geq \frac{1}{4}q^{n(n-1)-m/2-1}(q^{4}-1)(q^2-1)\geq 2^{5}\cdot 15\cdot 3>192,\]
a contradiction.
If $m$ is odd, then by Eq.~(\ref{eq-po-}), we have $m\leq n$. Again, since $p$ is an odd primitive prime of $q^m-1$, one has $p\bigm| (q^{m}-1)$.
Then $p^3\leq q^m(q^m-1)^2\leq q^m(q^n+1)(q^{2n-2}-1)$, and then
\[\frac{|T|}{p^3}\geq \frac{q^{n(n-1)}(q^n+1)\prod_{i=1}^{n-1}(q^{2i}-1)}{4q^m(q^n+1)(q^{2n-2}-1)}\geq \frac{1}{4}q^{n(n-1)-m}(q^{4}-1)(q^2-1)\geq 2^{6}\cdot 15\cdot 3>192,\]
a contradiction.

\medskip
\f{\bf Subcase~2.5.}\ $T=^2B_2(q)$ with $q=2^{2t+1}$ and $t\geq1$.\medskip

In this case, by checking Table~\ref{table=3}, we have $e_L(m)=1$ with $m=1$ or $4$.
Also, recall that $|T|=q^2(q^2+1)(q-1)$, and that $p$ is an odd primitive prime of $q^m-1$.
If $m=1$, then $p\parallel (q-1)$ and $p^3<q^2(q-1)$, and then
$$192>\frac{|T|}{p^3}> q^2+1,$$
which means that $q\leq13$.
Thus, $q=8$ and $p=7$. However, $^2B_2(8)$ has no subgroups of order $7^2$, also a contradiction.
If $m=4$, then by Claim~2, $p\bigm| (q^2+1)$.
Since $p\nmid q(q-1)$ and $p^2\parallel |T|$, one has $p^2\parallel (q^2+1)$, forcing
$$192>\frac{|T|}{p^3}> \frac{q^2(q-1)}{p}> \frac{q^2(q-1)}{3}.$$
This implies that $q\leq 8$, and so $q=8$, and $p=5$ or $13$, contradicting to $p^2\parallel (q^2+1)$.

\medskip
\f{\bf Subcase~2.6.}\ $T=^3D_4(q)$.\medskip

In this case, by checking Table~\ref{table=2}, we have $e_L(m)=1$ with $m=12$ and $e_L(m)=2$ with $m=1,2,3$, or $6$.
Also, recall that $|T|=q^{12}(q^8+q^4+1)(q^6-1)(q^2-1)$ and that $p$ is an odd primitive prime of $q^m-1$.
If $e_L(m)=1$ with $m=12$, then $p^2\parallel (q^8+q^4+1)$ and $p\bigm| (q^6+1)$. It follows that $p\bigm| (q^4-q^2+1)$ and so $p\leq q^4$. Thus, $p^3\leq q^{12}$. If $e_L(m)=2$ with $m=1,2,3$, or $6$, then $p\leq q^3+1<q^4$ and so $p^3<q^{12}$. Thus, we always have
\[\frac{|T|}{p^3}\geq (q^8+q^4+1)(q^6-1)(q^2-1)>192,\]
a contradiction.

\medskip
\f{\bf Subcase~2.7.}\ $T=G_2(q) (q>2)$.\medskip

In this case, by checking Table~\ref{table=2}, we have $e_L(m)=1$ with $m=3$ or $6$, and $e_L(m)=2$ with $m=1$ or $2$.
Also, $|T|=q^6(q^6-1)(q^2-1)$. If $m=3$ or $6$, then since $p$ is an odd primitive prime of $q^m-1$ and $p^2\parallel |T|$, one has $p^2\parallel (q^m-1)$ and so $p^2\leq q^3+1$. In particular, $p^3\leq q^3(q^3+1)$. If $m=1$ or $2$, then $p\leq q+1$ and then $p^3\leq (q+1)^3<q^3(q^3+1)$. Since $q>2$, one has
\[\frac{|T|}{p^3}\geq q^3(q^3-1)(q^2-1)>192,\]
a contradiction.

\medskip
\f{\bf Subcase~2.8.}\ $T=^2G_2(q)$ with $q=3^{2t+1}$ and $t\geq1$. \medskip

In this case, $|T|=q^3(q^3+1)(q-1)$. By checking Table~\ref{table=2}, we have $e_L(m)=1$ with $m=1, 2$ or $6$.
If $m=1$, then $p\bigm| (q-1)$, and so $p<q$.
This gives rise to $$\frac{|T|}{p^3}>(q^3+1)(q-1)\geq (3^9+1)(3^3-1)>192,$$
which is impossible.
If $m=2$ or $6$, then since $p\nmid q(q-1)$ and $p^2\parallel |T| $, one has $p^2\parallel (q^3+1)$, and so
$$\frac{|T|}{p^3}>\frac{q^3(q-1)}{p}\geq\frac{3^9(3^3-1)}{3}>192,$$
a contradiction.

\medskip
\f{\bf Subcase~2.9.}\ $T=F_4(q)$.\medskip

In this case, we have $|T|=q^{24}(q^{12}-1)(q^8-1)(q^6-1)(q^2-1)$.
By checking Table~~\ref{table=2}, we have $e_L(m)=1$ with $m=8$ or $12$ and $e_L(m)=2$ with $m=3,4$ or $6$.
Note that $p$ is an odd primitive prime of $q^m-1$ and that $p^2\parallel |T|$.
Then $p^2\parallel (q^m-1)$.
If $m\neq3$, then Claim~2 means that $p^2\parallel (q^{m/2}+1)$, and so $p^2\leq q^6+1$, forcing $p^3\leq q^6(q^6+1)$.
If $m=3$, then $p\bigm| (q^3-1)$ implies that $p^3\leq (q^3+1)^3\leq q^6(q^6+1)$. Therefore, we have
\[\frac{|T|}{p^3}\geq q^{18}(q^{6}-1)(q^8-1)(q^6-1)(q^2-1)>192,\]
a contradiction.

\medskip
\f{\bf Subcase~2.10.}\ $T=^2F_4(q)$.\medskip

In this case, we have $|T|=q^{12}(q^6+1)(q^4-1)(q^3+1)(q-1)$ and $q=2^{2t+1}$.
By checking Table~\ref{table=2}, we have $e_L(m)=1$ with $m=6,8$ or $12$ and $e_L(m)=2$ with $m=1,2$ or $4$.
Note that $p$ is an odd primitive prime of $q^m-1$ and that $p^2\parallel |T|$.
If $m=6,8$ or $12$,
then $p^2\parallel (q^{m/2}+1)$ and so $p^3\leq q^{m/2}(q^{m/2}+1)\leq q^6(q^6+1)$. If $m=2$ or $4$, then $p\parallel (q^{m/2}+1)$, and if $m=1$, then $p\parallel (q-1)$. In particular, $p^3\leq (q^2+1)^3\leq q^6(q^6+1)$.  Therefore, we have
\[\frac{|T|}{p^3}\geq q^{6}(q^4-1)(q^3+1)(q-1)>192,\]
a contradiction.

\medskip
\f{\bf Subcase~2.11.}\ $T=E_6(q)$.\medskip

In this case, $|T|=\frac{1}{(3,q-1)}q^{36}\prod_{i\in\{2,5,6,8,9,12\}}(q^i-1)$.
By checking Table~\ref{table=2}, we have $e_L(m)=1$ with $m=5,8,9$ or $12$, and $e_L(m)=2$ with $m=4$ or $6$.
Note that $p$ is an odd primitive prime of $q^m-1$ and that $p^2\parallel |T|$.
Then $p^2\parallel (q^m-1)$, and by Claim~2, $p^2\parallel (q^{m/2}+1)$ when $m$ is even.
If $m=8$ or $12$, then $p^3\leq q^{m/2}(q^{m/2}+1)\leq q^6(q^6+1)$, while if $m=4$ or $6$, then $p^3\leq (q^3+1)^3\leq q^6(q^6+1)$.
If $m=5$, then $p^2\parallel \frac{q^5-1}{q-1}=q^4+q^3+q^2+q+1$,
and if $m=9$, then $p^2\parallel \frac{q^9-1}{q^3-1}=q^6+q^3+1$.
As a result, we always have $p^3\leq q^6(q^6+q^3+1)$, and so
\[\frac{|T|}{p^3}\geq \frac{1}{3}q^{30}\prod_{i\in\{2,3,5,6,8,12\}}(q^i-1)>192,\]
a contradiction.

\medskip
\f{\bf Subcase~2.12.}\ $T=^2E_6(q)$.\medskip

In this case, we have $|T|=\frac{1}{(3,q+1)}q^{36}\prod_{i\in\{2,5,6,8,9,12\}}(q^i-(-1)^i)$.
By checking Table~\ref{table=2}, we have $e_L(m)=1$ with $m=8,10,12$ or $18$, or $e_L(m)=2$ with $m=3$ or $4$.
Note that $p$ is an odd primitive prime of $q^m-1$ and that $p^2\parallel |T|$. If $m=8,10,12$ or $18$, then $p^2\parallel (q^{m/2}+1)$ and so $p^3\leq q^{m/2}(q^{m/2}+1)\leq q^9(q^9+1)$. If $m=4$, then $p\bigm| (q^{2}+1)$, and so $p^3\leq (q^2+1)^3\leq q^9(q^9+1)$. If $m=3$, then $p\bigm| (q^{2}+q+1)$, and so $p^3\leq (q^2+q+1)^3\leq q^9(q^9+1)$. As a result, we always have $p^3\leq q^9(q^9+1)$, and so
\[\frac{|T|}{p^3}\geq \frac{1}{3}q^{27}\prod_{i\in\{2,5,6,8,12\}}(q^i-(-1)^i)>192,\]
a contradiction.

\medskip
\f{\bf Subcase~2.13.}\ $T=E_7(q)$.\medskip

In this case, $|T|=\frac{1}{(2,q-1)}q^{63}\prod_{i\in\{2,6,8,10,12,14,18\}}(q^i-1)$.
By checking Table~\ref{table=2}, we have $e_L(m)=1$ with $m=5,7,8,9,10,12,14$ or $18$ and $e_L(m)=2$ with $m=4$.
Note that $p$ is an odd primitive prime of $q^m-1$ and that $p^2\parallel |T|$. If $m=8,10,12,14$ or $18$, then $p^2\parallel (q^{m/2}+1)$ and so $p^3\leq q^{m/2}(q^{m/2}+1)\leq q^9(q^9+1)$. If $m=5$ or $7$, then $p^2\parallel (q^m-1)$, and so $p^3\leq q^7(q^7-1)\leq q^9(q^9+1)$. If $m=4$, then $p\bigm| (q^2+1)$, then $p^3\leq (q^2+1)^3\leq q^9(q^9+1)$. As a result, we always have $p^3\leq q^9(q^9+1)$, and so
\[\frac{|T|}{p^3}\geq \frac{1}{2}q^{54}\prod_{i\in\{2,6,8,9,10,12,14\}}(q^i-1)>192,\]
a contradiction.

\medskip
\f{\bf Subcase~2.14.}\ $T=E_8(q)$.\medskip

In this case, $|T|=q^{120}\prod_{i\in\{2,8,12,14,18,20,24,30\}}(q^i-1)$.
By checking Table~\ref{table=2}, we have $e_L(m)=1$ with $m=7,9,14,15,18,20,24$ or $30$, and $e_L(m)=2$ with $m=5, 8,10$ or $12$.
Note that $p$ is an odd primitive prime of $q^m-1$ and that $p^2\parallel |T|$. If $m=14,18,20,24$ or $30$, then $p^2\parallel (q^{m/2}+1)$ and so $p^3\leq q^{m/2}(q^{m/2}+1)\leq q^{15}(q^{15}+1)$. If $m=8, 10$ or $12$, then $p\parallel (q^{m/2}+1)$ and so $p^3\leq (q^{m/2}+1)^3\leq (q^6+1)^3\leq q^{15}(q^{15}+1)$. If $m=5$, then $p\parallel (q^{5}-1)$ and so $p^3\leq (q^5-1)^3\leq q^{15}(q^{15}+1)$. If $m=7$ or $9$, then $p^2\parallel (q^{m}-1)$ and so $p^3\leq q^m(q^m-1)< q^{15}(q^{15}+1)$. As a result, we always have $p^3\leq q^{15}(q^{15}+1)$, and so
\[\frac{|T|}{p^3}\geq q^{105}\prod_{i\in\{2,8,12,14,18,20,24,15\}}(q^i-1)>192,\]
a contradiction.\hfill\qed

The conclusion of Lemma~\ref{lem-simple-p2} can be immediately derived from Lemmas~\ref{lem-non-Lie-type} and~\ref{lem-Lie-type}.

\end{document}